\documentclass[a4paper,reqno]{amsart}

\usepackage[utf8]{inputenc}
\usepackage[english,french]{babel}
\usepackage{geometry}       
\usepackage{amsmath,amsfonts,amssymb,amsthm}
\usepackage{mathrsfs}
\usepackage{mathabx} 

\usepackage[linktocpage]{hyperref} 
\hypersetup{colorlinks=true,linkcolor=blue,citecolor=blue,filecolor=magenta,urlcolor=blue}

\usepackage{multiabstract}

\usepackage{color}
\usepackage[pdftex]{graphicx}
\usepackage{tikz, pgfplots,tikz-cd}

\newtheorem{thm}{Theorem}[section]  
\newtheorem{propr}{Property}[section]
\newtheorem{propo}{Proposition}[section]
\newtheorem{cor}{Corollary}[section]
\newtheorem{lem}{Lemma}[section]

\newtheorem*{conj*}{Conjecture}

\newtheorem*{prob*}{Problem}
\theoremstyle{remark}
\newtheorem{rem}{Remark}[section]
\newtheorem*{notation*}{Notation}
\theoremstyle{definition}
\newtheorem{defn}{Definition}[section]
\newtheorem{ex}{Example}[section]

\makeatletter
\let\c@lem\c@thm
\let\c@propr\c@thm
\let\c@propo\c@thm
\let\c@conj\c@thm
\let\c@prob\c@thm
\let\c@cor\c@thm
\let\c@rem\c@thm
\let\c@defn\c@thm
\let\c@notation\c@thm
\let\c@ex\c@thm
\makeatother

\def\makeautorefname#1#2{\expandafter\def\csname#1autorefname\endcsname{#2}}

\makeautorefname{thm}{Theorem}%
\makeautorefname{lem}{Lemma}%
\makeautorefname{rem}{Remark}%
\makeautorefname{notation}{Notation}%
\makeautorefname{defn}{Definition}%
\makeautorefname{ex}{Example}%
\makeautorefname{propr}{Property}
\makeautorefname{propo}{Proposition}%
\makeautorefname{cor}{Corollary}%
\makeautorefname{conj}{Conjecture}%
\makeautorefname{prob}{Problem}%

\newcommand\NN{{\mathbb N}}
\newcommand\ZZ{{\mathbb Z}}
\newcommand\QQ{{\mathbb Q}}
\newcommand\RR{{\mathbb R}}
\newcommand\CC{{\mathbb C}}

\newcommand\QQbar{\overline{\mathbb Q}}
\newcommand\PR{{\mathbb P}_{\mathbb R}}
\newcommand\PC{{\mathbb P}_{\mathbb C}}

\newcommand\xx{{\mathbf{x}}}

\newcommand\Acal{{\mathcal A}}

\newcommand\Hcal{{\mathcal H}}
\newcommand\Ical{{\mathcal I}}

\newcommand\Pkz{{\mathcal P}_{\text{\sc kz}}}

\newcommand\saq{{\mathcal{SA}}_{\RRalg}}

\newcommand\dint{{\mathrm{d}}}

\newcommand\ocube{\mathbb{B}{}^\infty_{o}}
\newcommand\cube{\overline{\mathbb{B}}{}^\infty_{o}}

\DeclareMathOperator{\vol}{vol}
\DeclareMathOperator{\sgn}{sgn}
\DeclareMathOperator{\dist}{dist}

\DeclareMathOperator{\id}{id}
\DeclareMathOperator{\codim}{codim}

\DeclareMathOperator{\Sing}{Sing}
\DeclareMathOperator{\Jac}{Jac}
\DeclareMathOperator{\alg}{alg}

\newcommand\RRalg{{\RR}_{\alg}}

\makeatletter
\def\@tocline#1#2#3#4#5#6#7{\relax
  \ifnum #1>\c@tocdepth 
  \else
    \par \addpenalty\@secpenalty\addvspace{#2}%
    \begingroup \hyphenpenalty\@M
    \@ifempty{#4}{%
      \@tempdima\csname r@tocindent\number#1\endcsname\relax
    }{%
      \@tempdima#4\relax
    }%
    \parindent\z@ \leftskip#3\relax \advance\leftskip\@tempdima\relax
    \rightskip\@pnumwidth plus4em \parfillskip-\@pnumwidth
    #5\leavevmode\hskip-\@tempdima
      \ifcase #1
       \or\or \hskip 1em \or \hskip 2em \else \hskip 3em \fi%
      #6\nobreak\relax
    \dotfill\hbox to\@pnumwidth{\@tocpagenum{#7}}\par
    \nobreak
    \endgroup
  \fi}
\makeatother

\usepackage{todonotes}

\usepackage{algorithm}
\usepackage[noend]{algpseudocode}

\usepackage{caption}
\makeatletter
  \def\vhrulefill#1{\leavevmode\leaders\hrule\@height#1\hfill \kern\z@}
\makeatother
\DeclareCaptionFormat{algor}{%
  \vhrulefill{1pt}\par\offinterlineskip\vskip1pt%
    \textbf{#1#2}#3\offinterlineskip\vhrulefill{0.5pt}}
\DeclareCaptionStyle{algori}{singlelinecheck=off,format=algor,labelsep=space}
\title{A semi-canonical reduction for periods of Kontsevich-Zagier}

\author{Juan Viu-Sos}
\address{
	IMPA - Instituto de Matem\'atica Pura e Aplicada,
	Estr. Dona Castorina, 110 - Jardim Bot\^anico,
	Rio de Janeiro - RJ, 22460-320, Brazil}
\urladdr{https://jviusos.github.io/}
\email{jviusos@math.cnrs.fr}
\thanks{The author is supported by a PNPD/CAPES postdoctoral grant and was also supported by a postdoctoral grant \#2016/14580-7 by \emph{Fundação de Amparo à Pesquisa do Estado de São Paulo} (FAPESP).}



\date{}

\begin{document}

\begin{abstracts}
	\abstractin{english}
	The $\QQbar$-algebra of periods was introduced by Kontsevich and Zagier as complex numbers whose real and imaginary parts are values of absolutely convergent integrals of $\QQ$-rational functions over $\QQ$-semi-algebraic domains in $\RR^d$. The Kontsevich-Zagier period conjecture affirms that any two different integral expressions of a given period are related by a finite sequence of transformations only using three rules respecting the rationality of the functions and domains: additions of integrals by integrands or domains, change of variables and Stokes formula.
	
	In this paper, we prove that every non-zero real period can be represented as the volume of a compact $\RRalg$-semi-algebraic set obtained from any integral representation by an effective algorithm satisfying the rules allowed by the Kontsevich-Zagier period conjecture.
	\bigskip
	
	\abstractin{french}
	La $\QQbar$-alg\`ebre des p\'eriodes fut introduite par Kontsevich et Zagier comme les nombres complexes dont les parties r\'eelle et imaginaire sont valeurs d'int\'egrales absolument convergentes de fonctions $\QQ$-rationnelles sur des domaines $\QQ$-semi-alg\'ebriques dans $\RR^d$. La conjecture des p\'eriodes de Kontsevich-Zagier affirme que si une p\'eriode admet deux repr\'esentations intégrales, alors elles sont reli\'ees par une suite finie d'op\'erations en utilisant uniquement trois r\`egles respectant la rationalit\'e des fonctions et domaines : sommes d'int\'egrales par int\'egrandes ou domaines, changement de variables et formule de Stokes.
	
	Dans cet article, nous d\'emontrons que toute p\'eriode r\'eelle non nulle peut \^etre repr\'esent\'ee comme le volume d'un ensemble $\RRalg$-semi-alg\'ebrique compact obtenu \`a partir de n'importe quelle repr\'esentation int\'egrale via un algorithme effectif en respectant les r\`egles permises par la conjecture des p\'eriodes de Kontsevich-Zagier.
\end{abstracts}

\selectlanguage{english}
\maketitle



\tableofcontents


\pagebreak

\section{Introduction}\label{sec:intro}

%

Introduced by Kontsevich and Zagier in their 2001 article~\cite{KonZag01}, \emph{periods} are a class of numbers which contains most of the important constants in mathematics. They are strongly related to transcendence in number theory \cite{Wald06}, Galois theory and motives (\cite{Andre04}, \cite{Andre12}, \cite{Ayoub15}), and differential equations \cite{FisherRivoal14}. We refer to \cite{Wald15} and \cite{Muller14} for an overview on the subject.\\

Let $\QQbar$ (resp. $\RRalg$) be the field of complex (resp. real) algebraic numbers. 
A \emph{period of Kontsevich-Zagier} (also called \emph{effective period}) is a complex number whose real and imaginary parts are values of absolutely convergent integrals of rational functions over domains in a real affine space given by polynomial inequalities both with coefficients in $\RRalg$ , i.e. absolutely convergent integrals of the form
\begin{equation}\label{eqn:int_form}
    \Ical(S,P/Q)=\int_S \frac{P(x_1,\ldots,x_d)}{Q(x_1,\ldots,x_d)}\cdot \dint x_1\wedge\ldots\wedge \dint x_d
\end{equation}
where $S\subset\RR^d$ is a $d$-dimensional $\RRalg$-semi-algebraic set and $P,Q\in\RRalg[x_1,\ldots,x_d]$ are coprime. We denote by $\Pkz$ the set of periods of Kontsevich-Zagier and by $\Pkz^\RR=\Pkz\cap\RR$ the set of real periods. These numbers are \emph{constructible}, in the sense that any period is directly associated with a set of integrands and domains of integration given by polynomials with rational coefficients. 
The set $\Pkz$ forms a constructible countable $\QQbar$-algebra and contains many transcendental numbers such as $\pi$. Other examples of periods are the multiple zeta values (MZV), i.e.
\[
    \zeta(s_1, \ldots, s_k) = \sum_{n_1 > n_2 > \cdots > n_k > 0} \ \frac{1}{n_1^{s_1} \cdots n_k^{s_k}}
\]
for $s_1,\ldots,s_k$ positive integers and $s_1>1$. These numbers, properties and representations are studied in a combinatorial way by expressing the above series as iterated integrals of two kind of very simple rational functions over simplices. See~\cite{Wald00} for an extensive overview on MZV.

\subsection{Two open problems for periods}

Fixing a period, it can be defined by many different integral representations. Thus, a first natural question is to determine how these different representations are related to each other.

In their article, Kontsevich and Zagier described two open problems following this direction:
\begin{enumerate}
    \item {\bf The Kontsevich-Zagier (KZ) period conjecture (\cite[Conjecture 1]{KonZag01}).}  If a real period admits two integral representations, then we can pass from one integral to the other by only using three operations (called the \emph{KZ-rules}): integral additions by domains or integrands, change of variables and the Stokes formula. Moreover, these operations should respect the class of the objects defined above.

    \item {\bf Equality algorithm (\cite[Problem~1]{KonZag01}).} The determination of an algorithm which allows one to determine whether two periods are equal or not.
\end{enumerate}

\subsection{A semi-canonical reduction}

Even though the previous definition of periods is explicit and elementary, it does not give us a precise idea of what a period is and neither how to deal with the relations among them. %
An idea to simplify the way in which we deal with periods is to rewrite the integral representation either as a \emph{volume of a domain with trivial differential form} or an \emph{integral over a fixed ``trivial'' domain}.

The first kind of such reductions was suggested by Kontsevich and Zagier in~\cite[p. 3]{KonZag01}, where they proposed to express any arbitrary period as the volume of a semi-algebraic set. Concerning the second one, Ayoub proved a relative version of the KZ-conjecture~\cite{Ayoub14,Ayoub15} by encoding the complexity of periods over the differential form, fixing a simple domain of integration. In particular, he proved that the algebra of periods can be defined as the complex values coming from integrals $\int_{[0,1]^d} f\left(\underline{z}\right) \dint \underline{z}$ where $f(\underline{z})$, with $\underline{z}=(z_1,\ldots,z_d)\in\CC^d$, is a holomorphic function over the unit polydisk $\mathcal{D}^d=\{|z_i|\leq1\}$.\\

In this paper, we prove that any non-zero period can be algorithmically reduced up to sign as the volume of a compact $\RRalg$-semi-algebraic set in $\RR^d$. Moreover, we give a constructive way to obtain such a reduction from any integral representation of the period, respecting the three operations of the KZ-conjecture and using classical tools in algebraic geometry, such as (algorithmic) \emph{resolution of singularities}.

\begin{thm}[\bf Semi-canonical reduction]\label{thm:semialg_form_thm}
  Let $p$ be a non-zero real period given in a certain integral form $\Ical(S,P/Q)$ in $\RR^d$ as in (\ref{eqn:int_form}). There exists an effective algorithm only using KZ-rules such that $\Ical(S,P/Q)$ can be rewritten as %
  \[
     \Ical(S,P/Q) = \sgn(p)\cdot\vol_m(K),
  \]
  where $K\subset\RR^m$ is a compact top-dimensional semi-algebraic set and $\vol_m(\cdot)$ is the canonical volume in $\RR^m$, for some $0<m\leq d+1$.%
\end{thm}

\begin{rem}
    By an \emph{algorithm} or a \emph{constructive procedure} we mean a finite explicit sequence of operations which produces an \emph{output} from a given \emph{input}, where each operation is described explicitly. Furthermore, an algorithm is called \emph{effective} if each operation can be effectively implemented on a machine.
\end{rem}

The effective algorithm of Theorem~\ref{thm:semialg_form_thm} is called a \emph{reduction algorithm}, and it is a refinement of Villamayor's algorithm~\cite{Vill89} for Hironaka's resolution of singularities, see Section~\ref{subsec:blowup_compact}. An explicit {\it pseudo-code} of this reduction is presented in Algorithm~\ref{alg:reduction} (see Appendix~\ref{sec:appendix} for more details).

\begin{rem}
    We can extend Theorem~\ref{thm:semialg_form_thm} for the whole set of periods $\Pkz\subset\CC$ considering representations of the real and imaginary part respectively. Such a representation for a period $p$ is called \emph{a geometric semi-canonical representation of $p$}. The existence of this kind of representation was assumed by Wan~\cite{Wan11} in order to develop a \emph{degree theory} for periods.
\end{rem}

As a direct consequence of Theorem~\ref{thm:semialg_form_thm}, we obtain the following result.

\begin{cor}\label{cor:periods_are_volumes}
    Any real period can be written up to sign as the volume of an affine compact semi-algebraic set.
\end{cor}

These results are part of the author's PhD thesis~\cite{ViuPhD}. Other results in this direction 
where obtained first by Yoshinaga~\cite[Sec.~3.2]{Yosh08} and later by Huber and
M\"uller-Stach in their recent book~\cite[Sec.~12.2]{HuberMullerStach:book}, both using the desingularization ideas for period integrals given by Belkale and Brosnan in~\cite{BelBro03} and adapted triangulations of semi-algebraic sets over varieties.

Motivated by how periods arise in quantum field theory, Belkale and Brosnan made use of resolution of singularities techniques to detail how $\Pkz$ is generated by integrals of a regular differential form over a top-dimensional relative cycle (with respect to a normal crossing divisor) in a smooth variety.

To our knowledge, Yoshinaga's work is the first geometric study of $\Pkz$ using volumes. In particular, he proved that $\Pkz$ is generated as an algebra by volumes of compact semi-algebraic sets in $\RR^d$. Using this, he developed an approximation theory for volumes of basic bounded semi-algebraic sets by Riemann sums, in order to exhibit a number which is not a period.

Huber and M\"uller-Stach's book covers most of the current knowledge about periods. One of the parts deals with the several existing ways of defining periods, showing that they all describe $\Pkz$. In particular, they prove that any period in $\Pkz^\RR$ can be written as the difference of two compact semi-algebraic sets of $\RR^d$. On the other hand, they also consider integrals of type $\int_S \omega$ where both $S$ and $\omega$ are not necessarily top-dimensional as a way of defining periods.

Our results cover the former descriptions of $\Pkz^\RR$ in terms of volumes constructed from integrals of type $\Ical(S,P/Q)$ (see Corollary~\ref{cor:pairs_semi_comp}). Moreover, we give a way to obtain a single compact semi-algebraic set expressing a period, providing an effective algorithm only using KZ-rules and avoiding triangulations of semi-algebraic sets.
\\

The word \emph{semi-canonical} refers to the fact that the resulting compact semi-algebraic set obtained by the reduction algorithm depends on the initial integral representation $\Ical(S,P/Q)$. In order to obtain geometrical information of a period coming from these semi-canonical representations, we need to deal with two phenomena:
\begin{itemize}
	\item {\bf Non-uniqueness of the dimension.} Given a period, one can obtain two representations in affine spaces of different dimensions. For example, $\pi^2$ can be obtained as the 4-dimensional volume of the Cartesian product of two copies of the unit disk and also as the $3$-dimensional volume of the set
	\[
		S_1=\left\lbrace (x,y,z)\in\RR^3 \mid x^2+y^2\leq1, 0\leq z((x^2+y^2)^2 + 1)\leq4\right\rbrace.
	\]
	
	\item {\bf Non-uniqueness for a fixed dimension.} We can find two compact semi-algebraic sets in some $\RR^d$ with the same volume, not being related by a linear change of variables. E.g.~taking the 2-dimensional volume of the unity semi-disk and the 2-dimensional volume of
	\[
		S_2=\left\lbrace (x,y)\in\RR^2 \mid 0<x<1, 0< y(1+x^2)<1\right\rbrace,
	\]
	we obtain $\pi/2$ in both cases.
\end{itemize}
The first issue can be fixed considering the \emph{minimal dimension} for which a period admits such a representation, leading to the notion of \emph{degree of a period}~\cite{Wan11}. As for the second one, more information about the nature of the compact semi-algebraic sets can be considered, e.g.~the notion of \emph{complexity of semi-algebraic sets} (see \cite[Sec. 4.5, p. 211]{BenRis90}). Despite these ambiguities, the semi-canonical reduction is a convenient tool to manipulate and compare different periods. In particular, this gives a way to approach the Kontsevich-Zagier period conjecture~\cite{CVS16}.

A similar connection between some periods and volumes was already established for sums of generalized harmonic series (see \cite{BeuKolkCal93}). However, the type of change of variables and domains which are used in loc. cit. does not belong to the KZ-rules.\\

The proof of Theorem~\ref{thm:semialg_form_thm} is based on the compactification of semi-algebraic sets and algorithmic resolution of singularities. Essentially, one has to deal with three main issues:
\begin{itemize}
	\item The first one comes from the framework of the KZ-conjecture, i.e.~only operations and constructions using the KZ-rules are allowed.
	
	\item The second one is to provide constructive methods at each step of the proof. This constraint does not appear in the formulation of the KZ-conjecture, but it is motivated by the problem of {\bf accessible identities}, i.e. identities between periods which can be obtained by a construction algorithm (see~\cite[Problem~1]{KonZag01}).

	\item The last one is more technical and it is related to the fact that we have to deal with affine compact semi-algebraic domains: we need to provide affine charts and partitions of domains which guarantee local compactness during the resolution process. Note that the arithmetic nature of the objects is preserved due to the behavior of resolution of singularities~\cite{Hir64}.
\end{itemize} 

As a general rule in our procedures, we give partitions of semi-algebraic sets cutting by hyperplanes, in order not to increase the complexity of the representation of the resulting semi-algebraic sets.\\

The outline of the paper is as follows. In Section~\ref{sec:sas_projective}, we construct a compactification of semi-algebraic sets by the natural inclusion into the real projective space $\PR^d$ by means of the \emph{projective closure} of a semi-algebraic set. Then, we resolve the poles at the boundary using resolution of singularities, in the same spirit as in~\cite{BelBro03}. However, we focus on the constructibility of the resolution, as well as on the way to construct partitions of the domain by affine compact sets. As a consequence, we prove that periods can be expressed as the difference of the volumes of two compact semi-algebraic sets (see Corollary~\ref{cor:pairs_semi_comp}). 
In Section~\ref{sec:diff_vol}, we complete the proof of Theorem~\ref{thm:semialg_form_thm} providing an explicit asymptotic method which allows us to write the difference of the volumes of two compact semi-algebraic sets $K_1$ and $K_2$ as the volume of a single one constructed algorithmically from $K_1$ and $K_2$. 
Examples of semi-canonical reductions of periods are presented in Section~\ref{sec:examples}. Finally, we derive our conclusions in Section~\ref{sec:conclusions}. A list of pseudo-codes implementing each of the preceding procedures is detailed in Appendix~\ref{sec:appendix}.

\bigskip

\noindent {\it Conventions.} 
    Throughout this article:
    \begin{enumerate}
        \item All the algebraic varieties are considered over the field of real algebraic numbers. We construct our theory from the real point of view, but most of the results about resolution of singularities can be obtained using classical algebraic geometry over algebraically closed fields by complexification of the varieties. In addition, we assume that the domains of integration $S$ are closed and \emph{regular}, i.e. $S$ coincides with the topological closure of its interior.
        
        \item The rational differential forms are considered forgetting the orientation: $(P/Q)(x_1,\ldots,x_d)\cdot |\dint x_1\wedge\ldots\wedge\dint x_d|$, i.e. we integrate a rational function over the Lebesgue measure on $\RR^d$. With a slight abuse of notation, we use $\dint x_1\wedge\ldots\wedge\dint x_d$ from now on.
    \end{enumerate}

\bigskip
\section{Semi-algebraic compactification of domains and resolution of poles}\label{sec:sas_projective}

The aim of this section is to algorithmically construct a representation of a period as integrals of regular rational functions over compact semi-algebraic sets, holding ambient dimension, and using partitions of domains and birational change of variables from another representation $\Ical(S,P/Q)$. Then, we derive that any period can be expressed as the difference of volumes of two compact semi-algebraic sets.

\subsection{Preliminaries about semi-algebraic geometry}
We recall basic definitions and properties about semi-algebraic sets and functions, following~\cite{BCR98}.

\begin{defn}
    A subset $S\subset\RR^d$ is called \emph{$\RRalg$-semi-algebraic} if it is can be described as
    \[
        S=\bigcup_{i=1}^s\bigcap_{j=1}^{r_i}\{f_{i,j} \ast_{i,j} 0\}
    \]
	where $f_{i,j}\in\RRalg[x_1,\ldots,x_d]$ and $\ast_{i,j}\in\{=,>\}$, for $i=1,\ldots,s$ and $j=1,\ldots,r_i$. A subset $S$ of a variety $W$ defined over $\RRalg$ is called \emph{$\RRalg$-semi-algebraic} if $S\cap U$ is $\RRalg$-semi-algebraic (considered as a subset of $\RR^d$) for every affine Zariski open subset $U\subset W$.
\end{defn}

In the following, we refer to $\RRalg$-semi-algebraic sets simply as \emph{semi-algebraic} sets.

\begin{propr}
    The semi-algebraic class is closed by finite unions, finite intersections and taking complements.
\end{propr}

\begin{propr}
    Let $S\subset \RR^{d+1}$ be a semi-algebraic set and $\pi:\RR^{d+1}\to \RR$ the projection on the first $d$ coordinates. Then $\pi(S)$ is a semi-algebraic subset of $\RR^d$.
\end{propr}

\begin{defn}
    Let $A\subset \RR^m$ and $B\subset \RR^n$ be two semi-algebraic sets. A mapping $f:A\to B$ is \emph{semi-algebraic} if its graph
    \[
        \text{Graph}(f) = \{(a,f(a))\in A\times B \mid a\in A\}
    \]
    is semi-algebraic in $\RR^{m+n}$.
\end{defn}

\begin{propr}
    Let $f:A\to B$ be a semi-algebraic mapping. Then:
    \begin{enumerate}
        \item The image and inverse image of semi-algebraic sets by $f$ are semi-algebraic.
        
        \item If $g:B\to C$ is a semi-algebraic mapping, then the composition $g\circ f$ is semi-algebraic.
        
        \item The $\RR$-valued semi-algebraic functions on a semi-algebraic set $A$ form a ring with addition and composition.
    \end{enumerate}
\end{propr}

\begin{ex}
	Piecewise polynomial and rational functions over semi-algebraic sets are semi-algebraic. For a semi-algebraic $\emptyset\neq A\subset \RR^d$, the \emph{distance function} to $A$, denoted by $\dist(x,A)$ and defined in $\RR^d$ is continuous semi-algebraic. Note that $\dist(x,A)$ vanishes over $\overline{A}$ and is positive elsewhere.
\end{ex}

\begin{propr}
    The semi-algebraic class is stable by taking the interior, closure and boundary.
\end{propr}

Following~\cite{BCR98}, we define the \emph{dimension} of a semi-algebraic set as the dimension of its Zariski closure. We denote by $\saq^d$ the set composed by all top-dimensional semi-algebraic subsets of $\RR^d$. Any open semi-algebraic set can be expressed as a finite union of open \emph{basic} semi-algebraic sets (\cite[Thm 2.7.2, p. 46]{BCR98}), i.e. semi-algebraic sets of the form $\{f_1>0,\ldots,f_s>0\}\subset\RR^d$, for some $f_1,\ldots,f_s\in\RRalg[x_1,\ldots,x_d]$.\\

For a semi-algebraic set $S$, we are interested in the study of the Zariski closure of $\partial S$, denoted by $\partial_z S$. In general, it is very difficult to give a representations of $\partial_z S$ in terms of the polynomials describing $S$. Using stratifications of semi-algebraic sets~\cite[Chapter 9]{BCR98}, we can give a decomposition of a semi-algebraic set of $S$ by open basic semi-algebraic sets of the form $B=\{f_1>0,\ldots,f_s>0\}$ up to zero-measure sets and, in this case, $\partial_z B\subset\left\lbrace\prod_{i=1}^s f_i=0\right\rbrace$.

\subsection{Projective closure of semi-algebraic sets and compact domains}\label{subsec:proj_closure_sas}

We are interested in the study of affine semi-algebraic sets and their compactification in the real projective space $\PR^d$. The main idea is to \emph{projectivize} our integration domains, and then decompose them into a collection of affine compact sets.

Denote by $[x_0:\ldots:x_d]$ the usual coordinates in $\PR^d$ and define the projective hyperplanes $\Hcal_{x_i}=\{x_i=0\}$. We consider the usual \emph{atlas} of $\PR^d$ given by $\{(U_{x_i},\varphi_{x_i})\}_{i=0}^d$, described by open Zariski sets $U_{x_i}=\PR^d \setminus \Hcal_{x_i}=\{x_i\neq0\}$, and the regular functions
\[
 \begin{array}{rccc}
    \varphi_{x_i}:& U_{x_i} & \longrightarrow & \RR^d\\
		  & [x_0:\ldots:x_d] & \longmapsto & \left(\frac{x_0}{x_i},\ldots,\frac{x_{i-1}}{x_i},\frac{x_{i+1}}{x_i},\ldots,\frac{x_d}{x_i}\right)%
 \end{array}
\]

\begin{rem}
    In the complex case, the projectivization of an algebraic set via homogenization is a classical tool to study algebraic varieties:  the topological closure of its inclusion coincides with the Zariski closure in $\PC^d$ by homogeneous polynomials. This is no more true in the real case: some extra points can appear in the real projective variety defined by homogenization and outside the topological closure of the affine part. E.g. the zero locus of $f(x,y)=(x+y)^2+1$ is empty in $\RR^2$, but that of $F(x,y,z)=(x+y)^2+z^2$ in $\PR^2$ consists of a double point.
\end{rem}

\begin{rem}\label{rem:sas_proj_desc}
Let $S$ be a semi-algebraic component in the first chart $U_{x_0}$ described by
\[
 S = \{(x_1,\ldots,x_d)\in\RR^d \mid p(x_1,\ldots,x_d)=0,\ q_i(x_1,\ldots,x_d)>0, i=1,\ldots,n\}.
\]
Its image in the other charts $\tilde{S}_j=\varphi_{x_j}\varphi_{x_0}^{-1}(S\setminus\{x_j\neq 0\})$ is also a semi-algebraic set and it can be expressed in local coordinates $(t_0,\ldots,\hat{t_j},\ldots, t_d)\in\RR^d$ as follows:
\[
  \tilde{S}_j=\left\lbrace t_0\neq0,\ t_0^{-d_p}P(t_0,\ldots, t_d)_{|t_j=1}=0,\ t_0^{-d_i}Q_i(t_0,\ldots, t_d)_{|t_j=1}>0, i=1,\ldots,n\right\rbrace,
\]
where $P$ and $Q_1,\ldots,Q_n$ are the homogenization of $p$ and $q_1,\ldots,q_n$, respectively, and $d_p=\deg p$, $d_i=\deg q_i$ for $i=1,\ldots,n$. Thus, $\tilde{S}_j$ splits into two disjoint semi-algebraic sets $\tilde{S}_j^{\pm}$ where:
\begin{align*}
 \tilde{S}_j^+ & =\left\lbrace t_0>0,\ P_{|t_j=1}=0,\ {Q_i}_{|t_j=1}>0,\ i=1,\ldots,n\right\rbrace,\\
 \tilde{S}_j^- & =\left\lbrace t_0<0,\ P_{|t_j=1}=0,\ (-1)^{d_i}{Q_i}_{|t_j=1}>0,\ i=1,\ldots,n\right\rbrace.
\end{align*}
Note that if $S$ is not contained in ${x_j=0}$, then at least one of the above sets is not empty.
\end{rem}

We define the \emph{projective closure} of a semi-algebraic set $S\subset\RR^d$ by $\overline{\varphi_{x_0}^{-1}S}$, i.e. the topological closure of the inclusion of $S$ into $\PR^d$ considering $\Hcal_{x_0}$ as the hyperplane at infinity. Note that the projective closure of a semi-algebraic set is a compact semi-algebraic set in $\PR^d$.\\%

Let $\ocube(r)$ (resp. $\cube(r)$) be the open (resp. closed) hypercube of radius $r>0$ centered at the origin in $\RR^d$ , i.e. $\ocube(r)=\{|x_i|< r\}$ (resp. $\cube(r)=\{|x_i|\leq r\}$). In the following, we give a useful decomposition of the real projective space $\PR^d$ as $d+1$ (closed) hypercubes centered at the origin of each usual chart, glued through their opposite faces.

\begin{propo}\label{propo:partition_projective}
    Let $\{C_i\}_{i=0}^d$ be the family of compact semi-algebraic sets in $\PR^d$ defined by $C_i=\overline{\varphi_{x_0}^{-1}V_i}$ where $V_0=\cube(1)$ and, for $1\leq i\leq d$, the set $V_i$ is the union of
    \[
	  \bigcap_{\substack{j=1\\ j\neq i}}^d \left\lbrace (x_1,\ldots,x_d)\in\RR^d \mid x_i-1\geq 0,x_i-x_j\geq 0,x_i+x_j\geq 0\right\rbrace
	\]
	and
	\[
	  \bigcap_{\substack{j=1\\ j\neq i}}^d \left\lbrace (x_1,\ldots,x_d)\in\RR^d \mid x_i+1\leq 0,x_i-x_j\leq 0,x_i+x_j\leq 0\right\rbrace.
	\]
	Then:
	\begin{enumerate}
		\item $C_i\subset U_{x_i}$ and $\varphi_{x_i}C_i=\cube(1)$, for any $0\leq i\leq d$.
	    \item $\PR^d=\bigcup_{i=0}^d C_i$.
	    \item The Zariski closure of $\bigcup_{i,j=0}^d (C_i\cap C_j)$ is the union of hyperplanes $\Acal=\{x_i^2-x_j^2=0 \mid 0\leq i<j\leq d\}$ of $\PR^d$.
	\end{enumerate}
\end{propo}

\begin{proof}
    By construction, the set $C_i$ is disjoint to the hyperplane $\Hcal_{x_i}=\{x_i=0\}$ in $\PR^d$, for any $i=0,\ldots,d$. %
    By a change of charts $\varphi_{x_i}\varphi_{x_0}^{-1}$ in $\PR^d$ induced by taking $\Hcal_{x_i}$ as hyperplane at infinity, we obtain that
	\[
	  \varphi_{x_i}\varphi_{x_0}^{-1}V_i=\bigcap_{j=1}^d \left\lbrace t_0\neq 0, -1\leq t_0\leq 1, -1\leq t_j\leq 1\right\rbrace,
	\]
	in local coordinates $(t_0,\ldots,\hat{t}_i,\ldots,t_d)$ in $\RR^d$. Taking the topological closure, one has that $\overline{\varphi_{x_i}\varphi_{x_0}^{-1}V_i}=\varphi_{x_i} C_i=\cube(1)$. %
	One can check that $\bigcup_{i=0}^d V_i=\RR^d$ (identified with $\PR^d\setminus\Hcal_{x_0}$), and it follows that the topological closure of each $V_i$ induces a covering of $\PR^d$. Finally, the intersection of two regions $C_i$ and $C_j$ is a $(d-1)$-dimensional semi-algebraic set contained in $\{x_i+x_j=0\}\cup\{x_i-x_j=0\}$, and this completes the proof.
	
\end{proof}

Using the above family of compact semi-algebraic sets once we have fix coordinates, we compactify our semi-algebraic domain of integration: first, passing through the projective space by projective compactification, and then decomposing it using $\{C_i\}_{i=0}^d$.

\begin{thm}\label{thm:sum_int_over_compacts}
    Let $S\in\saq^d$ be an open semi-algebraic set and $\omega=P/Q\cdot \dint x_1\wedge\ldots\wedge \dint x_d$ with $P/Q\in\RRalg\left(x_1,\ldots,x_d\right)$ such that the integral $\Ical(S,P/Q)$ converges absolutely. Then there exist a $(d-1)$-dimensional semi-algebraic set $X\subset\RR^d$, a partition $S=X\cup S_0\cup\cdots\cup S_d$, and a collection $\{\varphi_i\}_{i=1}^d$ of birational morphisms $\varphi_i:\RR^d\to \RR^d$ such that
    \[
	  \Ical(S,P/Q) = \sum_{i=0}^{d} \int_{\varphi_i^{-1}S_i} \varphi_i^* \omega,
    \]
    where $\varphi_i^{-1}S_i$ is bounded and $\varphi_i^* \omega$ is a rational $d$-form defined in the interior of $S_i$ for any $i=0,\ldots,d$. Moreover, this procedure is effectively algorithmic and it only depends on the representation of $S$.
\end{thm}

\begin{proof}
Define $S_0=S\cap\ocube(1)$ and $\varphi_0=\id_{\RR^d}$. For $i=1,\ldots, d$, we fix a hyperplane of the form $\{x_i=1\}$ for local coordinates $(x_1,\ldots,x_d)$ in $\RR^d$, and we consider $V_i$ the unbounded semi-algebraic region described in Proposition~\ref{propo:partition_projective} with respect to this hyperplane. %
Defining $S_i=S\cap \mathring{V}_i$ and after a change of charts $\varphi_{x_i}\varphi_{x_0}^{-1}$ in $\PR^d$ by taking $\Hcal_{x_i}$ as hyperplane at infinity, we obtain that
\[
    \varphi_{x_i}\varphi_{x_0}^{-1}S_i\subset\varphi_{x_i}C_i=\cube(1),
\]
which is a bounded semi-algebraic set in local coordinates $(t_0,\ldots,\hat{t}_i,\ldots,t_d)$ in $\RR^d$. By construction, it follows that $X=S\setminus\bigcup_{i=0}^dS_i$ is a $(d-1)$-dimensional semi-algebraic set. Finally, by decomposition of domains (up to measure zero sets) and change of variables, one has  
\[
	\int_S \omega = \sum_{i=0}^{d} \int_{S_i} \omega = \sum_{i=0}^{d} \int_{\varphi_i^{-1}S_i} \varphi_i^* \omega.
\]
\end{proof}

\begin{cor}\label{cor:compact_domain}
    Any period can be represented as a sum of absolutely convergent integrals of rational functions in $\RRalg(x_1,\ldots,x_d)$ over compact semi-algebraic sets, obtained algorithmically and respecting the KZ-rules from any integral representation. 
\end{cor}

\subsection{Resolution of singularities and affine compactness}\label{subsec:blowup_compact}

By Theorem~\ref{thm:sum_int_over_compacts}, we can assume that we are only dealing with integrals of type $\Ical(S,P/Q)$ over a compact domain $S\in\saq^d$. 
Our aim now is to ``take the volume under the graph of integration'' in order to obtain volumes of compact semi-algebraic sets. However, this is not always achievable, due to possible poles appearing at the boundary of $S$. In addition, the projectivization procedure described in last section produces potential new poles along the different coordinate axes.

Note that the above situation does not arise in the 1-dimensional case, since poles cannot appear in the boundary of $S\subset\RR$ whenever $\Ical(S,P/Q)$ is absolutely convergent, and the change of variables over the projective line $\PR^1$ removes automatically the pole of order 2 appearing at the point of infinity (see Example~\ref{ex:pi1_dim1}). %
We also assume that $P/Q$ is not constant, otherwise our result holds by a linear change of variables.\\

We use techniques from resolution of singularities in order to remove possible poles and obtain integrands which are regular in the boundary of the semi-algebraic domain.

\begin{thm}[Embedded Resolution of Singularities,~{\cite[Thm.~2.4]{BravoEncinasVillamayor:simple_desing}}]\label{thm:res_sing}
    Let $W_0$ be a smooth variety defined over a field of characteristic zero and consider $X$ a closed reduced subvariety of $W_0$. There exists a finite sequence
    \begin{equation}\label{eqn:res_sing}
            (W_0,X_0) \stackrel{\pi_1}{\longleftarrow} (W_1,X_1\cup E_1) \stackrel{\pi_2}{\longleftarrow} (W_2,X_2\cup E_1\cup E_2) \ldots  \stackrel{\pi_r}{\longleftarrow} (W_r,X_r\cup E_1\cup\ldots\cup E_r)
    \end{equation}
 where:
 \begin{enumerate}
     \item $W_{k-1} \stackrel{\pi_j}{\longleftarrow} W_k$ are proper birational maps between smooth varieties, given by \emph{blow-ups} over a smooth center $Z_{k-1}\subset Z_k$.
     
     \item The composite $W_{0} \stackrel{\pi}{\longleftarrow} W_r$ is a proper birational map such that $W_0\setminus\Sing X_0\simeq W_r\setminus\bigcup_{i=1}^r E_i$.
     
     \item The \emph{strict transform} $X_r=\overline{\pi^{-1}(X_0\setminus\Sing X_0)}$ is a regular subvariety and has normal crossings with the \emph{exceptional hypersurface} $\bigcup_{i=1}^r E_i$ in $W_r$.
 \end{enumerate}
 Such a map $W_{0} \stackrel{\pi}{\longleftarrow} W_r$ is called an \emph{embedded resolution of $X$ in $W_0$}.
\end{thm}
Resolution of singularities in characteristic zero was first proved by Hironaka in~\cite{Hir64}. This process is \emph{efficiently algorithmic} as a consequence of Villamayor's works and his collaborators~\cite{Vill89,BravoEncinasVillamayor:simple_desing}, who give a way to choose smooths centers at each step. A resolution of singularities algorithm was first implemented in {\tt Maple} and {\tt Singular}~\cite{Sing} by Bodn\'ar and Schicho~\cite{BodSchi001,BodSchi002}, and the current implementation in {\tt Singular} was made by Fr\"{u}hbis-Kr\"{u}ger and Pfister\footnote{See~\url{https://www.singular.uni-kl.de/Manual/4-0-3/sing_1454.htm\#SEC1529} for more information.}. An overview on this algorithm and recent methods can be found in~\cite{Fruhbis07,BlancoFruhbis12}.

\begin{rem}\label{rmk:resolution}
	Let $f\in\RRalg[x_1,\ldots,x_d]$ be a non-constant polynomial and  $X=\{a\in\RR^d \mid f(a)=0\}$. Hironaka's desingularization theorem above constructs a proper birational map $\pi:W\to\RR^d$ by a finite sequence of blow-ups over smooth centers, where $W$ is a %
	smooth variety, %
	and $\pi$ restricts into an isomorphism $W\setminus\pi^{-1}\Sing X\simeq \RR^d\setminus\Sing X$.%
	
	Considering the family of exceptional hypersurfaces $\{E_1,\ldots,E_r\}$ of the resolution $\pi:W\to\RR^d$ and denoting by $E_0$ the strict transform by $\pi$, there exist a collection of pairs of non-negative integers $\{(N_i,\nu_i)\}_{i=0}^r$, called the \emph{numerical data} of the resolution, such that the divisors associated to the pull-back of $f$ and the canonical differential $d$-form by $\pi$ are of the form $\sum_{i=0}^r N_i E_i$ and $\sum_{i=0}^r (\nu_i-1) E_i$ in $W$, respectively. Thus, $N_i$ and $\nu_i-1$ are non-negative numbers which correspond to the multiplicities of $f\circ\pi$ and $\pi^*(\dint x_1\wedge\cdots\wedge\dint x_d)$ over $E_i$, respectively. For the family of smooth hypersurfaces $\{E_0, E_1,\ldots, E_r\}$, to have \emph{normal crossings} means that they look locally as a union of coordinate hyperplanes, i.e. for any point $a\in W$ verifying $(f\circ\pi)(a)=0$, there exist local coordinates $(y_1,\ldots,y_d)$ of $W$ centered in $a$ and $f_1,\ldots,f_r\in\RRalg[y_1,\ldots,y_d]$ such that
	\begin{enumerate}
	 \item $E_i$ has local equation $f_i=0$, for $0\leq i\leq r$.
	 
	 \item $\left(\dint f_1\right)_{|_0},\ldots,\left(\dint f_r\right)_{|_0}$ are linearly independent.
	 
	 \item There exists $g,h\in\RRalg[y_1,\ldots,y_d]$ satisfying $g(0),h(0)\neq0$ and
	  \[
		(f\circ\pi)=g\cdot\prod_{k=1}^r f_{i_k}^{N_{i_k}}\quad\text{and}\quad \pi^*\left(\bigwedge_{i=1}^d \dint x_i\right)=h\cdot\prod_{k=1}^r f_{i_k}^{\nu_{i_k}-1}\cdot\bigwedge_{i=1}^d \dint y_i,
	  \]
	for some $1\leq i_1,\ldots, i_r\leq d$.
	\end{enumerate}
	Furthermore, locally near $a$ we can express
	\[
		(f\circ\pi)=\varepsilon\cdot\prod_{k=1}^r y_{i_k}^{N_{i_k}}\quad\text{and}\quad \pi^*\left(\bigwedge_{i=1}^d \dint x_i\right)=\eta\cdot\prod_{k=1}^r y_{i_k}^{\nu_{i_k}-1}\cdot\bigwedge_{i=1}^d \dint y_i,
	\]
	for some $1\leq i_1,\ldots, i_r\leq d$ and $\varepsilon,\eta$ real analytic functions with $\varepsilon(0),\eta(0)\neq0$. See e.g.~\cite[Ch.~3~and~11]{Igu00} for more details.
\end{rem}

\begin{rem}
    Since any connected algebraic variety $W$ is covered by charts $\{(U_i,\varphi_i)\}_{i\in I}$ given by open Zariski sets and regular morphisms, one can compute an integral over $W$ restricting to any chart $U$, i.e. for any measurable set $D\subset W$ one has that $\int_D \omega=\int_{D\cap U} \omega_{| U}$.
\end{rem}

Let $W_0$ be a smooth $\RRalg$-variety. 
For a semi-algebraic set $S\subset W_0$ and a top-dimensional rational differential form $\omega$ in $W_0$, denote by $\partial_z S$ the Zariski closure of $\partial S$ and by $Z(\omega)$ and $P(\omega)$ the zero and pole locus of $\omega$ in $W_0$, respectively. Define the $\RRalg$-subvariety of $W_0$ given by $X_{S,\omega}=\partial_z S\cup Z(\omega)\cup P(\omega)$ and consider $\pi:W_r\to W_0$ an embedded resolution of $X_{S,\omega}$ obtained by Theorem~\ref{thm:res_sing}. The \emph{strict transform of $S$} is the semi-algebraic set of $W_r$ given by $\widetilde{S}=\overline{\pi^{-1}(S\setminus\Sing X_{S,\omega})}$. %
One can check that $\partial_z \widetilde{S}$ is a $\RRalg$-subvariety of $\pi^{-1}(\partial_z S)$. Moreover, $\partial_z \widetilde{S}$ is a union of irreducible components of $\pi^{-1}(\partial_z S)$, and $\widetilde{S}$ is equal to $\pi^{-1} (S)$ up to zero measure semi-algebraic sets.  

The main idea is to use the embedded resolution of singularities over $X_{S,\omega}$ to send the poles of the differential form in $\Ical(S,P/Q)$ ''sufficiently far away`` from $\partial S$ and then to give an algorithmic expression of the above as a sum of integrals of the same type over compact sets. The latter is achieved from the following geometric criterion about convergence of rational integrals over semi-algebraic sets.%

\begin{propo}\label{propo:separation_poles}
    Let $W_0$ be a smooth real algebraic variety defined over $\RRalg$. Let $S\subset W_0$ be a compact semi-algebraic set in $W_0$ and $\omega$ a top-dimensional rational differential form in $W_0$. 
    Then, the integral $\int_S \omega$ converges absolutely if and only if there exists a finite sequence of blow-ups $\pi=\pi_r\circ\cdots\circ\pi_1:W_r\to W_0$ over smooth centers as in (\ref{eqn:res_sing}) such that $\widetilde{S}\cap P(\pi^*\omega)=\emptyset$, where $\widetilde{S}$ is the strict transform of $S$ by $\pi$.
    
\end{propo}

\begin{proof}
	Suppose that $\int_S \omega$ converges absolutely. Note that $P(\omega)$ does not intersect the interior of $S$ in this case. Consider $X_{S,\omega}\subset W_0$ and $\pi:W_r\to W_0$ an embedded resolution of $X_{S,\omega}$ as above. Let $a$ be a point in $\partial\widetilde{S}\subset X_{S,\omega}$. Following Remark~\ref{rmk:resolution}, we know that there exist local coordinates $(y_1,\ldots,y_d)$ around $a$ in $W_r$ with $d=\dim W_0$ such that %
	we have the following local expressions: 
	$$\partial_z \widetilde{S}:\ \alpha\cdot\prod_{k=1}^r y_{i_k}^{a_{i_k}}=0\quad\text{and}\quad \pi^*\omega=\beta\cdot\prod_{k=1}^r y_{i_k}^{b_{i_k}} \cdot\bigwedge_{i=1}^d \dint y_i,$$
	for some $1\leq i_1,\ldots, i_r\leq d$, with $a_{i}\in\ZZ_{\geq0}$ and $b_{i}\in\ZZ$~for any $i\in\{0,\ldots,r\}$, and $\alpha,\beta$ are real analytic functions which do not vanish at the origin. Considering $\widetilde{S}_\varepsilon = \widetilde{S}\cap\mathbb{B}{}^\infty_{a}(\epsilon)$ in these local coordinates, for sufficiently small $\epsilon>0$, the integral $\int_{\widetilde{S}_\epsilon}|\pi^*\omega|$ can be expressed as a linear combination with non-negative coefficients of integrals of type
	\begin{equation}\label{eqn:integral_bu-proof}
		\int_{\Gamma_\epsilon } |\beta|\cdot\prod_{k=1}^r |y_{i_k}|^{b_{i_k}},
	\end{equation}
	where $\Gamma_\epsilon=\{0<s_iy_i<\epsilon \mid i=1,\ldots,d\}$ for some $s_i\in\{-1,1\}$. It follows that the integral (\ref{eqn:integral_bu-proof}) converges if and only if the exponents $b_{i_k}$ are all non-negative. This is equivalent to assert that $a\not\in P(\pi^*\omega)$. The reciprocal follows directly, since $\int_S\omega=\int_{\widetilde{S}}\pi^*\omega$ and $\pi^*\omega$ is regular over the compact set $\widetilde{S}$.
    
\end{proof}

Similar results can be found in~\cite[Propo.~4.2]{BelBro03} and in~\cite[Lemma~12.2.4]{HuberMullerStach:book}. Here, we focus in the case when $W_0=\RR^d$. It is worth noticing that we do not need to give a complete embedded resolution of $X_{S,\omega}$, but only to consider a finite sequence of blow-ups over smooth centers which separates $\partial\widetilde{S}$ from the pole locus of the pull-back of the differential form. The former implies that the Zariski closure of $\partial S \cap Z(\omega)\cap P(\omega)$ cannot be top-dimensional in $\RR^d$ whenever we are dealing with periods, because in this case the integral $\Ical(S,P/Q)$ becomes divergent.

\begin{rem}
    The integers $\{b_i\}_{i=0}^r$ appearing in the proof of Proposition~\ref{propo:separation_poles} can be expressed as 
    \[
       b_{i}=N_{i}^{P} - N_{i}^{Q} + \nu_{i} - 1,
    \]
    where $N_{i}^{P}$ and $N_{i}^{Q}$ are the multiplicities at $a$ of $P\circ\pi$ and $Q\circ\pi$ over the divisor $E_{i}$, respectively, for any $0\leq i\leq r$.
\end{rem}

\begin{rem}\label{rem:bu_decomposition}
	Each step of the resolution process in Proposition~\ref{propo:separation_poles} corresponds to an intermediary blow-up with its corresponding strict transform, denoted $\widetilde{S}$ by abuse of notation. In the following, we describe a natural decomposition of $\widetilde{S}$ by affine compact semi-algebraic pieces.
	
	Let $\pi: W\to\RR^d$ be the blow-up of $\RR^d$ at some smooth center $Z$. It is known that $W$ is a closed smooth $d$-dimensional $\RRalg$-subvariety of $\RR^d\times\PR^m$, with $m=\codim Z$. An atlas $\{(V_i,\psi_i)\}_{i=0}^{m}$ of $W$ is induced by the usual decomposition of $\PR^m=\bigcup_{i=0}^m U_{x_i}$ as follows: one defines  $V_i=W\cap(\RR^d\times U_{x_i})$, and both the isomorphism $\psi_i: V_i\to\RR^d$ and the birational map $\varphi_i=\pi\circ\psi_i^{-1}:\RR^d\to\RR^d$ can be obtained using Gröbner basis from the equations of $Z$. See e.g.~\cite{Hauser:7_histories_bu} for a nice introduction about blow-ups and~\cite[Sec.~4.3]{BodSchi001} for computational details about the latter.

	Consider the partition by closed hypercubes $\PR^m=\bigcup_{i=0}^m C_i$ given in Proposition~\ref{propo:partition_projective} and define the partition $\widetilde{S}=\bigcup_{i=0}^m\widetilde{S}_i$, where $\widetilde{S}_i=\widetilde{S}\cap \left(\RR^d\times C_i\right)$. It follows that $\widetilde{S}_i$ is compact, semi-algebraic and contained in $V_i$ by construction. Thus, $\psi_i(\widetilde{S}_i)\subset\RR^d$ is also compact and belongs to $\saq^d$. Note that $\widetilde{S}_i\cap \widetilde{S}_j$ has dimension at most $d-1$, for any $j\neq i$.
\end{rem}

\begin{cor}\label{cor:sum_int_comp}
    Any period can be represented as a sum of absolutely convergent integrals of rational functions in $\RRalg(x_1,\ldots,x_d)$ over compact semi-algebraic sets, such that each rational function is regular over the corresponding integration domain. Moreover, such a representation can be obtained in an effective algorithmic way respecting the KZ-rules from any other integral representation. 
\end{cor}

\begin{proof}
    Let $\Ical(S,P/Q)$ be an absolute convergent integral over $S\subset\RR^d$ and denote $\omega=P/Q\cdot\dint x_1\wedge\ldots\wedge\dint x_d$. By Corollary~\ref{cor:compact_domain}, we can assume that $S$ is compact. By Proposition~\ref{propo:separation_poles}, there exists a finite sequence of blow-ups $\pi=\pi_r\circ\cdots\circ\pi_1:W_r\to W_0=\RR^d$ over smooth centers $\{Z_k\}_{k=0}^{r-1}$ such that the pullback $\pi^*\omega$ is regular over the strict transform  $\widetilde{S}$. For any $1\leq k\leq r$, define the number $m_k=\codim Z_{k-1}$ and take the list of discrete sets $\{M_k\}_{k=1}^r$ defined by $M_k=\{0,\ldots,m_1\}\times\cdots\times\{0,\ldots,m_k\}$.  %
    We recursively apply the construction in Remark~\ref{rem:bu_decomposition} after the first blow-up $\pi_1: W_1\to\RR^d$. 
    At each subsequent blow-up $\pi_k: W_k\to W_{k-1}$ over the center $Z_{k-1}$, consider the atlas 
    $\{(V_{i_1,\ldots,i_k},\psi_{i_1,\ldots,i_k}) \mid (i_1,\ldots,i_k)\in M_k\}$ 
    of $W_k$ with %
    \[
    	W_k=\bigcup_{(i_1,\ldots,i_k)\in M_k} V_{i_1,\ldots,i_k},\qquad 
    	    \begin{tikzcd}[row sep=large, column sep = large]
    	    	\psi_{i_1,\ldots,i_k}: V_{i_1,\ldots,i_k} \arrow{r}{\text{isom.}} & \RR^d
    	    \end{tikzcd}
    	    ,
    \]
    and such that the map $\bigcup_{i_k=0}^{m_k} V_{i_1,\ldots,i_{k-1},i_{k}}\to V_{i_1,\ldots,i_{k-1}}\simeq\RR^d$ is a blow-up of the affine space representing $\pi_k$, for any $(i_1,\ldots,i_{k-1})\in M_{k-1}$. Considering the strict transform of $S$ at each step, we obtain a finite sequence
    \begin{equation*}
                \left(\RR^d,S\right) \stackrel{\pi_1}{\longleftarrow}
                \cdots \stackrel{\pi_k}{\longleftarrow} \left(W_k, \bigcup_{(i_1,\ldots,i_k)\in M_k} \widetilde{S}_{i_1,\ldots,i_k}\right) \stackrel{\pi_{k+1}}{\longleftarrow} \cdots  \stackrel{\pi_r}{\longleftarrow} \left(W_r, \bigcup_{(i_1,\ldots,i_r)\in M_r} \widetilde{S}_{i_1,\ldots,i_r}\right)=\left(W_r,\widetilde{S}\right),
    \end{equation*}
    such that $\widetilde{S}_{i_1,\ldots,i_k}$ is a compact semi-algebraic subset of $V_{i_1,\ldots,i_k}$ and  $$%
    \widetilde{S}_{i_1,\ldots,i_k}=\bigcup_{i_{k+1}=0}^{m_{k+1}} \pi_{k+1}(\widetilde{S}_{i_1,\ldots,i_{k+1}}),$$ for any $(i_1,\ldots,i_k)\in M_k$.
	The above partition of $\widetilde{S}$ induces a decomposition of $S$ by compact sets in $\saq^d$, i.e.
	\[
		S = \bigcup_{(i_1,\ldots,i_r)\in M_r} S_{i_1,\ldots,i_{r}}\qquad\text{with}\qquad  S_{i_1,\ldots,i_{r}} = \pi(\widetilde{S}_{i_1,\ldots,i_{r}}),
	\]
	verifying that $S_{i_1,\ldots,i_{r}}\cap S_{j_1,\ldots,j_{r}}$ has dimension at most $d-1$ for any different subindices. 
    Following this decomposition and considering the birational maps $\varphi_{i_1,\ldots,i_{r}}=\pi\circ(\psi_{i_1,\ldots,i_{r}})^{-1}:\RR^d\to\RR^d$, we obtain a sequence of KZ-operations:
    \[
        \Ical(S,P/Q)=\sum_{(i_1,\ldots,i_{r})\in M_r}\int_{(\varphi_{i_1,\ldots,i_{r}})^{-1}S_{i_1,\ldots,i_{r}}}(\varphi_{i_1,\ldots,i_{r}})^*\omega=\sum_{(i_1,\ldots,i_{r})\in M_r}\Ical(T_{i_1,\ldots,i_{r}},P_{i_1,\ldots,i_{r}}/Q_{i_1,\ldots,i_{r}})
    \]
    where $T_{i_1,\ldots,i_{r}}=(\varphi_{i_1,\ldots,i_{r}})^{-1}S_{i_1,\ldots,i_{r}}\in\saq^d$ is compact and $P_{i_1,\ldots,i_{r}},Q_{i_1,\ldots,i_{r}}\in\RRalg[x_1,\ldots,x_d]$ are coprime polynomials verifying that the zero locus of $Q_{i_1,\ldots,i_{r}}$ does not intersect $T_{i_1,\ldots,i_{r}}$, for any $(i_1,\ldots,i_{r})\in M_r$.
    
\end{proof}

\begin{rem}
	In the above proof, we construct a decomposition of $S$ by considering every chart and map at every blow-up, i.e. the entire \emph{resolution tree} given by $\{(V_\alpha,\varphi_\alpha)\}_{\alpha\in M}$. In practice, we can stop blowing-up in a chart $V_{\alpha}$ when the associated compact set $\widetilde{S}_{\alpha}$ does not intersect the pole locus.
\end{rem}

\begin{rem}
 	Despite the fact that resolution of singularities is algorithmic, the previous construction is not \emph{efficient} in practice. However, this can be improved when $S\subset\RR^2$ since in the 2-dimensional case we only need to resolve a finite set of points. Moreover, there is a more efficient method to decompose our domain and choose the charts of each blow-up by using the algebraic tangent cone of the $\partial_z S$ at the different poles. %
 	This method is explained in the author PhD thesis~\cite[Sec.~II.3]{ViuPhD} and it is exemplified throughout Section~\ref{sec:examples} (see Example~\ref{ex:pi1_dim2} for a brief explanation).
\end{rem}

Considering the above constructions we obtain the following result, cf.~\cite[Thm.~12.2.1]{HuberMullerStach:book}.

\begin{cor}
\label{cor:pairs_semi_comp}
Let $p\in\Pkz^\RR$ be expressed as an absolutely convergent integral of the form $\Ical(S,P/Q)$. Then $p$ can be expressed from $\Ical(S,P/Q)$ as
\[
  p=\vol_{d+1}(K_1) - \vol_{d+1}(K_2),
\]
where $K_1, K_2$ are compact $(d+1)$-dimensional semi-algebraic sets in $\RR^{d+1}$. Moreover, this process is effectively algorithmic and respects the KZ-rules.
\end{cor}

\begin{proof}
    Assume that $p\neq0$. Up to zero measure sets, we can give a partition of $S$ depending on the sign of the rational function $\frac{P}{Q}(x_1,\ldots,x_d)$ in $\RR^d$ as follows.
    \[
    	\Ical(S,P/Q)=\Ical(S^+,P/Q)-\Ical(S^-,-P/Q),
    \]
    where $S^\pm = \left\lbrace(x_1,\ldots,x_d)\in S\ \middle|\ \sgn\left((\frac{P}{Q})(x_1,\ldots,x_d)\right)=\pm 1\right\rbrace$. Note that both integrals have finite positive values since $\Ical(S,P/Q)$ is absolutely convergent. By Corollary~\ref{cor:sum_int_comp}, we can express them as:
	\[
		\Ical(S^\pm,P/Q)=\sum_{i=1}^{n_\pm}\Ical(S_i^\pm,P_i^\pm/Q_i^\pm),
	\]
    where $S_i^\pm\in\saq^d$ is compact and $P_i^\pm/Q_i^\pm\in\RRalg (x_1,\ldots,x_d)$ is reduced and regular over $S_i^\pm$, for any $i=1,\ldots,n_\pm$. Note that $P_i^\pm/Q_i^\pm$ does not change of sign over $S_i^\pm$.
    Considering the above integrals as the volume of the region delimited by the graph of $P_i^\pm/Q_i^\pm$ over $S_i^\pm$, we perform change of variables obtaining the following expression:
    \[
    	\Ical(S^\pm,P/Q)=\sum_{i=1}^{n_\pm}\int_{K_i^\pm}  1\ \dint t\dint x_1\cdots \dint x_d,
    \]
    where
    \begin{align*}
    	K_i^+ & =\left\lbrace (t,x_1,\ldots,x_d)\in\RR_{+}\times S_i^+\ \middle|\ t\leq\frac{P_i^+}{Q_i^+}(x_1,\ldots,x_d) \right\rbrace,\\
    	K_i^- & =\left\lbrace (t,x_1,\ldots,x_d)\in\RR_{+}\times S_i^-\ \middle|\ t\geq\frac{P_i^-}{Q_i^-}(x_1,\ldots,x_d) \right\rbrace.
    \end{align*}
	Both $K_i^+$ and $K_i^-$ are compact semi-algebraic sets. It remains to prove that $K_i^\pm\in\saq^{d+1}$. Defining $H_i^+=t\cdot Q_i^+-P_i^+\in\RRalg[t,x_1,\ldots,x_d]$, we have that $\left\lbrace t<({P_i^+}/{Q_i^+})(x_1,\ldots,x_d)\right\rbrace$ is expressed as the union of
	\[
	 \lbrace H_i^+(t,x_1,\ldots,x_d)<0\rbrace\cap\lbrace Q_i^+(x_1,\ldots,x_d)>0\rbrace,
	\]
	and
	\[
	 \quad\lbrace H_i^+(t,x_1,\ldots,x_d)>0\rbrace\cap\lbrace Q_i^+(x_1,\ldots,x_d)<0\rbrace.
	\]
    It follows that $K_i^+\in\saq^{d+1}$ since semi-algebraic domains are stable by finite union and intersection. Analogously, one has that $K_i^-\in\saq^{d+1}$.
    Since the sets $K_i^\pm$ are compact in $\RR^d$, there exists a sequence of $\RRalg$-translations $\left(\phi_i^\pm\right)_{i=1}^{n_\pm}$ in $\RR^{d+1}$ such that $\phi_i^\pm(K_i^\pm)\cap\phi_j^\pm(K_j^\pm)=\emptyset$, for any $i\neq j$. Defining $K_1=\bigcup_{i=1}^{n_+} \phi_i^+(K_i^+)$ and $K_2=\bigcup_{i=1}^{n_-} \phi_i^-(K_i^-)$, the result holds.
    
\end{proof}

\bigskip
\section{Difference of volumes and Riemann approximations}\label{sec:diff_vol}

We finish the proof of Theorem~\ref{thm:semialg_form_thm} detailing an algorithmic construction to obtain a single compact semi-algebraic set representing a period. Let $p$ be a non-zero period expressed as the difference of two volumes as in Corollary~\ref{cor:pairs_semi_comp}.%

\subsection{Partition by Riemann sums}\label{subsec:Riemann_sums}

Without loss of generality, we assume that $p$ is positive and that $0<\vol_d(K_2) < \vol_d(K_1)$. %
We use a similar geometric construction to that in~\cite[Sec.~3.4]{Yosh08} about inner and outer Riemann sums of compact semi-algebraic sets.\\

Since $K_1$ and $K_2$ are bounded subsets of $\RR^d$, we assume that there exists a positive integer $r>0$ such that both sets are contained in the hypercube $[0,r]^d$. In the following, we construct a partition of both $K_1$ and $K_2$ using rational subcubes. Let $n$ be a positive integer and define the parametrized family of cubes subdividing $[0,r]^d$ as follows:
\[
    C_n(k_1,\ldots,k_d)=\left[\frac{k_1}{n}r,\frac{k_1+1}{n}r\right]\times\ldots\times\left[\frac{k_d}{n}r,\frac{k_d+1}{n}r\right],
\]
where $0\leq k_1,\ldots,k_d\leq n$ are integers. Denote by $\mathring{C}_n(k_1,\ldots,k_d)$ the interior of the above cube.

The latter forms a parametrized partition of $[0,r]^d$ by cubes of size $(r/n)^d$. For $i=1,2$, consider such cubes intersecting $K_i$, i.e.
\[
  \widehat{\Delta}_n^{(i)}=\{(k_1,\ldots k_d)\in\{0,\ldots,n\}^d\mid C_n(k_1,\ldots k_d)\cap K_i\neq\varnothing\},
\]
and those which are contained in $K_i$:
\[
  \widecheck{\Delta}_n^{(i)}=\{(k_1,\ldots k_d)\in\{0,\ldots,n\}^d\mid C_n(k_1,\ldots k_d)\subset K_i\}.
\]
Denote by $\widehat{\delta}_i(n)$ and $\widecheck{\delta}_i(n)$ the cardinal of $\widehat{\Delta}_n^{(i)}$ and $\widecheck{\Delta}_n^{(i)}$, respectively. Since compact semi-algebraic sets are Borel sets, the following relation holds:
\begin{equation}\label{eqn:limits_delta}
  \lim_{n\to\infty}\widehat{\delta}_i(n)\cdot\left(\frac{r}{n}\right)^d=\lim_{n\to\infty}\widecheck{\delta}_i(n)\cdot\left(\frac{r}{n}\right)^d=\vol_d(K_i),\qquad i=1,2.
\end{equation}

\begin{lem}\label{lem:deltas_ineq}
  There exists a positive integer $n_0$ such that for any $N\geq n_0$, we have that $\widehat{\delta}_2(N)< \widehat{\delta}_1(N)$ and $\widecheck{\delta}_2(N)<\widecheck{\delta}_1(N)$.
\end{lem}

\begin{proof}
For any $n\in\NN$, we have the following outer approximation of $K_i$ by elements of $\widehat{\Delta}_n^{(i)}$:
\[
  0<\vol_d(K_i)\leq\widehat{\delta}_i(n)\cdot\left(\frac{r}{n}\right)^d,\quad i=1,2.
\]
By (\ref{eqn:limits_delta}), there exists a positive integer $\widehat{n}_0$ such that, for any $N\geq \widehat{n}_0$,
\[
  0<\vol_d(K_2)\leq\widehat{\delta}_2(N)\cdot\left(\frac{r}{N}\right)^d<\vol_d(K_1)\leq\widehat{\delta}_1(N)\cdot\left(\frac{r}{N}\right)^d.
\]
Thus, we obtain the relation $\widehat{\delta}_2(N)<\widehat{\delta}_1(N)$. 
The same holds for inner approximations using $\widecheck{\Delta}_n^{(i)}$, obtaining an analogous $\widecheck{n}_0$. Taking $n_0=\max\{\widehat{n}_0,\widecheck{n}_0\}$, the result holds.

\end{proof}

\begin{lem}\label{lem:deltas_ineq_inclusion}
  There exists a positive integer $n_0$ such that for any $N\geq n_0$, we have that $\widehat{\delta}_2(N)\leq \widecheck{\delta}_1(N)$.
\end{lem}

\begin{proof}
 For any $n\in\NN$, we decompose:
\[
  \widecheck{\delta}_1(n) - \widehat{\delta}_2(n) = (\widehat{\delta}_1(n) - \widehat{\delta}_2(n)) - (\widehat{\delta}_1(n) - \widecheck{\delta}_1(n)).
\]
Multiplying by $\left(r/n\right)^d$ and taking limits, we obtain that:
\begin{align*}
  \lim_{n\to\infty}(\widehat{\delta}_1(n) - \widehat{\delta}_2(n))\left(\frac{r}{n}\right)^d & = \vol_d(K_1)-\vol_d(K_2) = p,\\
  \lim_{n\to\infty}(\widehat{\delta}_1(n) - \widecheck{\delta}_1(n))\left(\frac{r}{n}\right)^d & = \vol_d(K_1)-\vol_d(K_1) = 0
\end{align*}
Recall that $p>0$ and $\widehat{\delta}_1(n) - \widecheck{\delta}_1(n)\geq 0$. Furthermore, $\widehat{\delta}_1(n) - \widehat{\delta}_2(n)>0$ for $n$ sufficiently large by Lemma~\ref{lem:deltas_ineq}. Thus, we have that:
\[
  \forall\varepsilon_0>0, \exists n_0\in\NN\ \text{s.t.}\ \forall N>n_0:\ (\widehat{\delta}_1(N) - \widecheck{\delta}_1(N))\left(\frac{r}{N}\right)^d<\varepsilon_0,
\]
and
\[
  \forall\varepsilon_1>0, \exists n_1\in\NN\ \text{s.t.}\ \forall N>n_1:\ \left|(\widehat{\delta}_1(N) - \widehat{\delta}_2(N))\left(\frac{r}{N}\right)^d - p\right|<\varepsilon_1.
\]
Taking $\varepsilon_1=1$ and $\varepsilon_0=C-\varepsilon_1=C-1$, there exists $n_2\in\NN$ such that for any $N>n_2$, one has that:
\[
  0 \leq (\widehat{\delta}_1(N) - \widecheck{\delta}_1(N))\left(\frac{r}{N}\right)^d < C-1 < (\widehat{\delta}_1(N) - \widehat{\delta}_2(N))\left(\frac{r}{N}\right)^d.
\]
Then, $\widehat{\delta}_2(N)\leq \widecheck{\delta}_1(N)$ for any $N>n_2$ and the result holds.

\end{proof}

\subsection{Construction of the difference set}\label{subsec:construct_diff_set}
We finish our procedure by constructing a compact $K\in\saq^d$ such that $|p| = \vol_d(K)$ from $K_1$ and $K_2$. The basic idea of this construction is to use inner and outer Riemann approximation by cubes in $K_1$ and $K_2$, respectively. Taking a sufficiently small rational length in the above parametrized partition, we can give a rearrangement of cubes such that the outer cubes of $K_2$ can be translated into the inner cubes of $K_1$.\\

By Lemma~\ref{lem:deltas_ineq_inclusion}, we know that there exists $n_0\in\NN$ such that $\widehat{\delta}_2(n_0)\leq \widecheck{\delta}_1(n_0)$. Consider the grid in $[0,r]^d$ defined by the boundary of all cubes in the partition:
\[
  W=\bigcup_{(k_1,\ldots,k_d)\in\{0,\ldots,n_0\}^d}\left\lbrace (x_1,\ldots,x_d)\in [0,r]^d \mid x_i=\frac{k_i}{n_0}r,\ 1\leq i\leq d\right\rbrace.
\]
This is a this zero measure subset in $[0,r]^d$, which is removed in the following.
\[
 H=[0,r]^d\setminus W=\bigcup_{(k_1,\ldots,k_d)\in\{0,\ldots,n_0\}^d} \mathring{C}_{n_0}(k_1,\ldots,k_d).
\]

Let $\Sigma_{n_0,d}=\Sigma(\{0,\ldots,n_0\}^d)$ be the symmetric group of $\{0,\ldots,n_0\}^d$. Any element $\tau=(\tau_1,\ldots,\tau_d)\in\Sigma_{n_0,d}$ naturally induces a bijective map in the above parametrization given by
\[
 \begin{array}{cccc}
	\psi_\tau:& \{0,\ldots,n_0\}^d & \longrightarrow & \{0,\ldots,n_0\}^d\\
		 & (k_1,\ldots ,k_d)			& \longmapsto	&	\left(\tau_1(k_1),\ldots ,\tau_d(k_d)\right)
  \end{array}.
\]

\begin{lem}\label{lem:diff_set_final}
 There exists a semi-algebraic map $\Psi:H\rightarrow H$ such that $\Psi$ is volume-preserving and verifies that $\Psi(H\cap K_2)\subset(H\cap K_1)$.
\end{lem}

\begin{proof}
 From Lemma~\ref{lem:deltas_ineq_inclusion}, there exists a $\sigma\in\Sigma_{n_0,d}$ verifying the following conditions:
 \begin{itemize}
  \item[\sf (P1)] $\psi_\sigma(\widehat{\Delta}_{n_0}^{(2)})\subset\widecheck{\Delta}_{n_0}^{(1)}$.
  \item[\sf (P2)] $\psi_\sigma=\id$ in $\{0,\ldots,n_0\}^d\setminus\widehat{\Delta}_{n_0}^{(2)}$.
 \end{itemize}
 The map $\psi_\sigma$ induces a bijective map $\Psi:H\rightarrow H$ which sends a point $(x_1,\ldots,x_d)$ in $\mathring{C}_{n_0}(k_1,\ldots,k_d)$ to the point
 \[
     \left(x_1 + \frac{\sigma_1(k_1) - k_1}{n_0}r ,\ldots, x_d + \frac{\sigma_d(k_d) - k_d}{n_0}r \right)\in\mathring{C}_{n_0}(\sigma_1(k_1),\ldots ,\sigma_d(k_d)).
 \]
 This map rearranges the open cubes in the partition of $[0,r]^d$ by translations induced by $\sigma$ and it is semi-algebraic by construction. The latter is a volume-preserving map and the conditions {\sf (P1)} and {\sf (P2)} imply that $\Psi(H\cap K_2)\subset(H\cap K_1)$.

\end{proof}

Finally, we can define $K$ as the closure over $\RR^d$ of $(H\cap K_1)\setminus\Psi(H\cap K_2)$ and this completes the proof of Theorem~\ref{thm:semialg_form_thm}. 
Note that the above process, which constructs the new compact semi-algebraic set $K$ from $K_1$ and $K_2$, is completely algorithmic, effective and respects the KZ-rules.

\bigskip
\section{Some examples of semi-canonical reduction}\label{sec:examples}
We detail some examples of semi-canonical reductions following the ideas of the effective reduction algorithm. Our aim is also to illustrate how one deals with the main issues of this procedure in concrete examples.

\subsection{A basic example: $\pi$}\label{subsec:pi}
 \begin{ex}\label{ex:pi1_dim1}
  A classical way to write $\pi$ as an integral is the following:
  \[
  	\Ical\left(\RR,1/(1+x^2)\right)=\int_{-\infty}^{+\infty}\frac{\dint x}{1+x^2}.
  \] 
  In order to obtain $\pi$ as the volume of a semi-algebraic set from the above, one first decomposes the real line in three pieces using the point arrangement $\mathcal{A}=\left\lbrace\{x=-1\},\{x=1\}\right\rbrace$. One obtains
  \[
  	\int_{-\infty}^{+\infty}\frac{\dint x}{1+x^2}=\int_{-1}^{1}\frac{\dint x}{1+x^2} + \int_{S}\frac{\dint x}{1+x^2},
  \]
  where $S=\{x^2-1>0\}$ is unbounded. Consider now the canonical inclusion of $S$ into the second chart $U_y=\{[x:y] \mid y\neq0\}$ of the projective line $\PR^1$. The change of charts with the first one $U_x=\{[x:y] \mid x\neq0\}$ gives a diffeomorphism $\phi$ of $\RR^*$ into itself expressed by $\phi(y)=1/y$, where $|\Jac(\phi)(y)|=1/y^2$ and $\phi^{-1}S=\{y\neq0, 1-y^2>0\}=(-1,1)\setminus \{0\}$. Then:
  \[
  	\int_{S}\frac{\dint x}{1+x^2}=\int_{\phi^{-1}S}\phi^*\left(\frac{\dint x}{1+x^2}\right)=\int_{(-1,1)\setminus \{0\}}\frac{y^2}{1+y^2}\cdot\frac{1}{y^2}\dint y=\int_{-1}^{1}\frac{\dint y}{1+y^2}.
  \]
  Thus, using partitions and rational change of variables given by $\phi$, we express:
  \[
  	\Ical\left(\RR,1/(1+x^2)\right)=\int_{-1}^{1}\frac{\dint x}{1+x^2} + \int_{S}\frac{\dint x}{1+x^2}=\int_{-1}^{1}\frac{\dint x}{1+x^2} + \int_{-1}^{1}\frac{\dint y}{1+y^2}.
  \]
  Taking the area under the graph in both integrals, followed by a symmetry with respect to the horizontal axis in the second integral, we obtain a semi-canonical reduction for $\pi$ (see Figure~\ref{fig:sc_red_pi}).
  \begin{align*}  
  	\pi & =\int_{\left\lbrace\begin{array}{c} -1\leq x\leq 1\\ 0\leq z(1+x^2)\leq 1 \end{array}\right\rbrace} \dint x \dint z + \int_{\left\lbrace\begin{array}{c} -1\leq y\leq 1\\ 0\leq u(1+y^2)\leq 1 \end{array}\right\rbrace} \dint y \dint u\\
  		& = \int_{\left\lbrace\begin{array}{c} -1\leq x\leq 1\\ 0\leq z(1+x^2)\leq 1 \end{array}\right\rbrace} \dint x \dint z + \int_{\left\lbrace\begin{array}{c} -1\leq y\leq 1\\ -1\leq u(1+y^2)\leq 0 \end{array}\right\rbrace} \dint u \dint y\\
  		& =\vol_2\left(\left\lbrace\begin{array}{c} -1\leq x\leq 1\\ -1\leq z(1+x^2)\leq 1 \end{array}\right\rbrace\right).
  \end{align*}

  \begin{figure}[ht]
  	\centering
  	\includegraphics[scale=0.4]{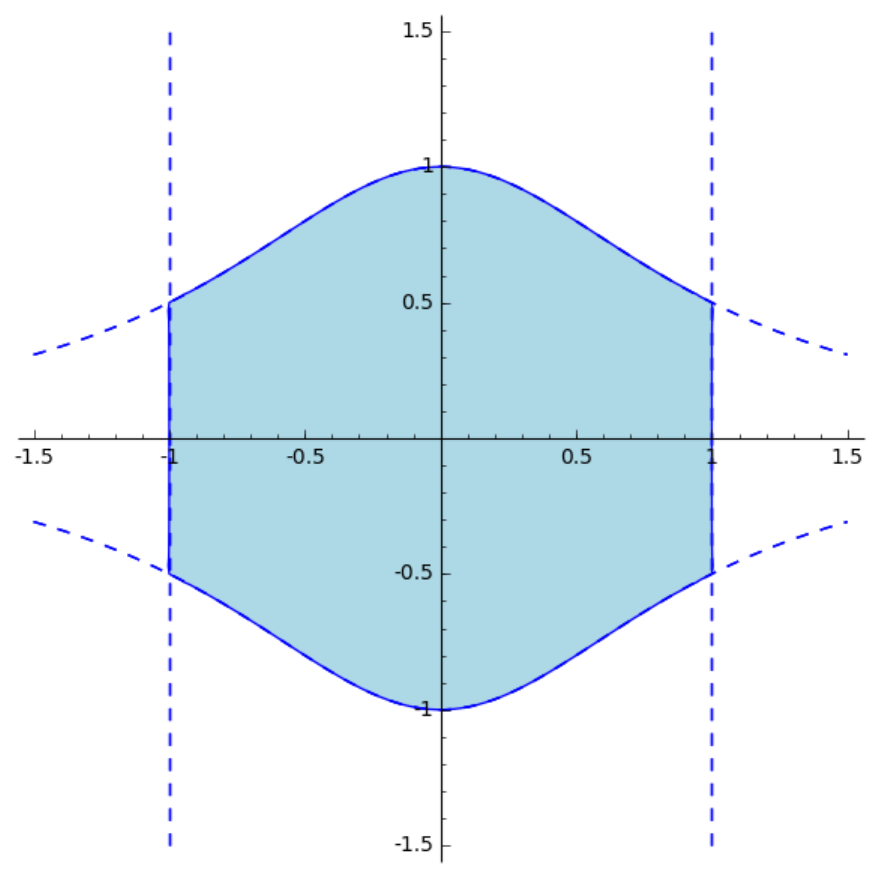}
  	\caption{A semi-canonical reduction of $\pi$ as a 2-dimensional volume of $K=\{-1\leq x\leq 1, -1\leq z(1+x^2)\leq 1\}$.}\label{fig:sc_red_pi}
  \end{figure}
\end{ex}

\begin{ex}\label{ex:pi1_dim2}
 Revisiting the previous example, we can represent a fraction of $\pi$ by the area of an unbounded 2-dimensional semi-algebraic set as follows. 
 \[
 	 \frac{\pi}{4} = \int_{1}^{\infty}\frac{1}{1+x^2}\dint x = \int_{D}\dint x\dint y
 \]
 with $D = \{x>1, 0<y(1+x^2)<1\}\subset\RR^2$ (see Figure~\ref{fig:pi_4}).
  
  \begin{figure}[ht]
  	\centering
  	\begin{tikzpicture}[scale=1.5]
	  	\colorlet{col1}{blue!40}
	  	\colorlet{col2}{red!40}
	  	\colorlet{col3}{violet!40}
	  	
	  	\begin{scope}[smooth,draw=gray!20]
	  	\filldraw [fill=col1,domain=1:4] plot  (\x, {1/(1+\x^2)}) -- (4,0)  -- (1,0) -- cycle;
	  	\draw[dashed,blue,domain=0:4,smooth] plot (\x, {1/(1+\x^2)});
	  	\draw[black,domain=1:4.1,smooth] plot (\x, {1/(1+\x^2)});
	  	\end{scope}
	  	\draw[->] (-0.25,0) -- (4.25,0) node [right] {$x$};
	  	\foreach \pos/\label in {0/$0$,1/$1$}
	  	\draw (\pos,0) -- (\pos,-0.1) (\pos cm,-3ex) node [anchor=base,inner sep=1pt]  {\label};
	  	\draw (-0.1,1) -- (.1,1);
	  	\node at (0,1.1) [right] {$1$};
	  	\node at (0,1.3) [left] {$y$};
	  	\node at (1.1,0.9) {$y=\frac{1}{1+x^2}$};
	  	\draw[->] (0,-0.25) -- (0,1.5);
	  	\node at (2,0.3) [above, font=\large] {$D$};
  	\end{tikzpicture}
  	\caption{The unbounded set $D = \{x>1, 0<y(1+x^2)<1\}$.}\label{fig:pi_4}
  \end{figure}
 
 Considering $D$ in the chart $U_z=\{[x:y:z] \mid z\neq0\}\subset\PR^2$ and taking $\{x=0\}\subset\PR^2$ as new line at infinity, we obtain a diffeomorphism  $\varphi$ of $\RR^2$ minus a line, given by $\varphi(x_1,y_1)=(1/x_1,y_1/x_1)$,  with associated jacobian determinant $ |\Jac(\varphi)|(x,y)=\frac{1}{x_1^3}$. 
 Thus,
 \begin{align*}
 	D_1 & =\varphi^{-1}D=\left\lbrace\frac{1}{x_1}>1,\ 0<\frac{y_1}{x_1}\left(1+\frac{1}{x_1^2}\right)<1\right\rbrace\\
 	    & = \left\lbrace 0<x_1<1,\ 0<y_1(1+x_1^2)<x_1^3\right\rbrace\\
 	    & = \left\lbrace 0<x_1<1,\ 0<y_1,\ 0<x_1^3-y_1(1+x_1^2)\right\rbrace,
 \end{align*}
 is bounded and it is related to the period as follows.
 \[
 	\Ical\left(D,1\right)=\int_{D}\dint x\dint y=\int_{D_1}\frac{\dint x_1\dint y_1}{x_1^3}.
 \]
 The jacobian contributes with a pole of order $3$ at the origin, which lies on the boundary of $D_1$.
 
\begin{figure}[ht]
   	\centering
   	\begin{tikzpicture}[scale=1.5]
   	    \colorlet{col1}{blue!40}
   	    \colorlet{col2}{red!40}
   	    \colorlet{col3}{violet!40}
   	
   	   \begin{scope}[smooth,draw=gray!20]
   			\draw[dashed,blue,domain=-0.5:1.75,smooth] plot (\x, {\x^3/(1+\x^2)});
   	        \filldraw [fill=col1,domain=0:1] plot (\x, {\x^3/(1+\x^2)})
   	            -- (1,0)  -- (0,0) -- cycle;
   	        \draw[black,domain=0:1,smooth] plot (\x, {\x^3/(1+\x^2)});
   	   \end{scope}
   	   \draw[->] (-0.25,0) -- (1.75,0) node [right] {$x_1$};
   	   \foreach \pos/\label in {0/,1/$1$}
   	        \draw (\pos,0) -- (\pos,-0.1) (\pos cm,-3ex) node [anchor=base,inner sep=1pt]  {\label};
   	   \node at (0,1.3) [left] {$y_1$};
   	   \draw[->] (0,-0.25) -- (0,1.5);
   	   \node at (1.5,0.15) [above, font=\large] {$D_1$};
   		\draw[semithick,dashed, color=red] (0,-0.4) -- (0,1.7);
   	\end{tikzpicture}
   	\caption{The domain $D_1=\left\lbrace 0<x_1<1,\ 0<y_1,\ 0<x_1^3-y_1(1+x_1^2)\right\rbrace$ and the pole locus of the integrand (red).}\label{fig:pi_4_D1}
   \end{figure}
 
 We are going to decrease the order of this pole by a sequence of blow-ups at the origin. Note that this order agrees with the intersection multiplicity at the origin of the curve $y_1(1+x_1^2)=x_1^3$ with the $x_1$-axis. %

 Performing a first blow-up at the origin, the situation in the first usual chart is described by the monomial transformation $\phi:\RR^2\to\RR^2$ given by $\phi(x_2,y_2)=(x_2,x_2y_2)$. In this new variables, we obtain that $\Ical\left(D_1,1/x_1^3\right)=\Ical\left(D_2,1/x_2^2\right)$, where $$D_2=\left\lbrace 0<x_2<1,\ \ 0<y_2,\ 0<x_2^2-y_2(1+x_2^2)\right\rbrace.$$
 It is worth noticing that $D_2$ is bounded and we have expressed our integral by only using one of the charts of the blow-up. This is due to the following geometric idea: taking a chart of the blow-up at the origin is essentially the same as choosing a $L$ in the pencil of lines passing through the origin of $\RR^2$ and sending $L$ to the new line at infinity. In this way, the pencil separates in parallel lines transverse to the exceptional divisor, which emerges from the origin to substitute the former $L$ in this new chart. Thus, the strict transform $\widetilde{D}$ of our domain $D$ remains bounded in the new chart whenever $L$ only meets $D$ at the origin and $L$ is not contained in the algebraic tangent cone $T_0(\partial_z D)$ at the origin. 
 
 In the previous transformation, one has $T_0(\partial_z D_1)=\{y_1=0\}$. We have chosen $L=\{x_1=0\}$, which is replaced by the exceptional divisor $E=\{x_2=0\}$ in the new chart with coordinates $(x_2,y_2)$.
 
 Repeating this process, pictured in Figure~\ref{fig:pi_4_res}, we obtain $\Ical\left(D_2,1/x_2^2\right)=\Ical\left(D_3,1/x_3\right)$, where $$D_3=\left\lbrace 0<x_3<1,\ \ 0<y_3,\ 0<-x_3^2y_3 + x_3-y_3\right\rbrace,$$ with $T_0(\partial D_3)=\{x_3-y_3=0\}$. Finally, we remove the pole at the origin and we obtain $\Ical\left(D_3,1/x_3\right)$ as the $2$-dimensional volume of the following set:
 $$D_4=\left\lbrace 0<x_4<1,\ \ 0<y_4,\ 0<-x_4^2y_4 -y_4 + 1\right\rbrace.$$
 
\begin{figure}[ht]
   	\centering
   	\begin{tikzpicture}
   	    \colorlet{col1}{blue!40}
   	    \colorlet{col2}{red!40}
   	    \colorlet{col3}{violet!40}
   	
   	\begin{scope}[xshift=-6cm]
   	 \begin{scope}[smooth,draw=gray!20]
   			\draw[dashed,blue,domain=-0.5:1.75,smooth] plot (\x, {\x^3/(1+\x^2)});
   	        \filldraw [fill=col1,domain=0:1] plot (\x, {\x^3/(1+\x^2)})
   	            -- (1,0)  -- (0,0) -- cycle;
   	        \draw[black,domain=0:1,smooth] plot (\x, {\x^3/(1+\x^2)});
   	   \end{scope}
   	   \draw[->] (-0.25,0) -- (1.75,0) node [right] {$x_1$};
   	   \foreach \pos/\label in {0/,1/$1$}
   	        \draw (\pos,0) -- (\pos,-0.1) (\pos cm,-0.4) node [anchor=base,inner sep=1pt, font=\footnotesize]  {\label};
   	   \node at (0,1.3) [left] {$y_1$};
   	   \draw[->] (0,-0.25) -- (0,1.5);
   	   \node at (0.7,0.5) [above, font=\large] {$D_1$};
   		\draw[semithick,dashed, color=red] (0,-0.4) -- (0,1.7);
   		 \draw[->] (3.25,0.75) -- (2.5,0.75);
   	\end{scope}
   	\begin{scope}[xshift=-2cm]
   	 \begin{scope}[smooth,draw=gray!20]
   			\draw[dashed,blue,domain=-0.5:1.75,smooth] plot (\x, {\x^2/(1+\x^2)});
   	        \filldraw [fill=col1,domain=0:1] plot (\x, {\x^2/(1+\x^2)})
   	            -- (1,0)  -- (0,0) -- cycle;
   	        \draw[black,domain=0:1,smooth] plot (\x, {\x^2/(1+\x^2)});
   	   \end{scope}
   	   \draw[->] (-0.25,0) -- (1.75,0) node [right] {$x_2$};
   	   \foreach \pos/\label in {0/,1/$1$}
   	        \draw (\pos,0) -- (\pos,-0.1) (\pos cm,-0.4) node [anchor=base,inner sep=1pt, font=\footnotesize]  {\label};
   	   \node at (0,1.3) [left] {$y_2$};
   	   \draw[->] (0,-0.25) -- (0,1.5);
   	   \node at (0.7,0.5) [above, font=\large] {$D_2$};
   		\draw[semithick,dashed, color=red] (0,-0.4) -- (0,1.7);
   		 \draw[->] (3.25,0.75) -- (2.5,0.75);
   	\end{scope}
   	\begin{scope}[xshift=2cm]
   	 \begin{scope}[smooth,draw=gray!20]
   			\draw[dashed,blue,domain=-0.5:1.75,smooth] plot (\x, {\x/(1+\x^2)});
   	        \filldraw [fill=col1,domain=0:1] plot (\x, {\x/(1+\x^2)})
   	            -- (1,0)  -- (0,0) -- cycle;
   	        \draw[black,domain=0:1,smooth] plot (\x, {\x/(1+\x^2)});
   	   \end{scope}
   	   \draw[->] (-0.25,0) -- (1.75,0) node [right] {$x_3$};
   	   \foreach \pos/\label in {0/,1/$1$}
   	        \draw (\pos,0) -- (\pos,-0.1) (\pos cm,-0.4) node [anchor=base,inner sep=1pt, font=\footnotesize]  {\label};
   	   \node at (0,1.3) [left] {$y_3$};
   	   \draw[->] (0,-0.25) -- (0,1.5);
   	   \node at (0.7,0.5) [above, font=\large] {$D_3$};
   		\draw[semithick,dashed, color=red] (0,-0.4) -- (0,1.7);
   		 \draw[->] (3.25,0.75) -- (2.5,0.75);
   	\end{scope}
   	\begin{scope}[xshift=6cm]
   	   \begin{scope}[smooth,draw=gray!20]
   			\draw[dashed,blue,domain=0:1.75,smooth] plot (\x, {1/(1+\x^2)});
   	        \filldraw [fill=col1,domain=0:1] plot (\x, {1/(1+\x^2)})
   	            -- (1,0)  -- (0,0) -- cycle;
   	        \draw[black,domain=0:1,smooth] plot (\x, {1/(1+\x^2)});
   	   \end{scope}
   	   \draw[->] (-0.25,0) -- (1.75,0) node [right] {$x_4$};
   	   \foreach \pos/\label in {0/,1/$1$}
   	        \draw (\pos,0) -- (\pos,-0.1) (\pos cm,-0.4) node [anchor=base,inner sep=1pt, font=\footnotesize]  {\label};
   	   \draw (-0.1,1) -- (.1,1);
   	   \node at (0,1) [left, font=\footnotesize] {$1$};
   	   \node at (0,1.3) [left] {$y_4$};
   	   \draw[->] (0,-0.25) -- (0,1.5);
   	   \node at (1,0.6) [above, font=\large] {$D_4$};
   	\end{scope}
   	
   	\end{tikzpicture}
   	\caption{Desingularization from $D_1$ to $D_4$ %
   	until removing the pole.}\label{fig:pi_4_res}
   \end{figure}
 
 Note that we have obtained a compact 2-dimensional domain from an unbounded one by our procedure, both representing the same period. However, this is not the case in general, since at this step one should take the volume under the integrand, which is usually not constant.
 \end{ex}

\begin{ex}[Another expression for $\pi$]
	Consider the period $\nu\in\Pkz^\RR$ described by the volume of the following unbounded two dimensional semi-algebraic set:
	\[
		\eta = \int_{S}\dint x\dint y,\quad\text{where}\quad S = \{x^4y^2-x+1<0\}\subset\{x>1\}.
	\]
	Taking the line $\{x=0\}$ in $\PR^2$ as line at infinity, we obtain a pole of order 3 in the chart with coordinates $(y,z)$. Moreover, 
	\[
		\int_{S}\dint x\dint y=\int_{S_0}\frac{\dint y\dint z}{z^3},\quad\text{with}\quad S_0 = \{z^6+y^2-z^2<0\}.
	\]
	Note that $S_0$ is contained in the upper semi-plane (see Figure~\ref{fig:pi_S1}) and $T_0(\partial S_0)=\{y^2=0\}$. Composing two blow-ups at the origin as before, we transform the integral by a diffeomorphism $\phi(y_2,z_2)=(y_2z_2^2,z_2)$ of $\RR^2\setminus\{z=0\}$ as follows:
	\[
	    \int_{S_0}\frac{\dint y\dint z}{z^3}=\int_{S_2}\frac{\dint y_2\dint z_2}{z_2},\quad\text{where}\quad S_2=\left\lbrace y_2^2+z_2^2-z_2<0 \right\rbrace.
	\]
	
	\begin{figure}[ht]
   	\centering
   	\includegraphics[scale=0.45]{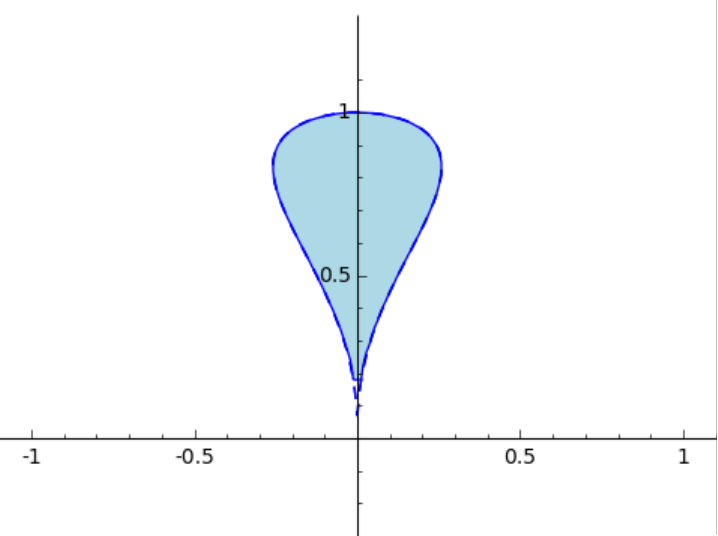}\hspace*{1cm}\includegraphics[scale=0.45]{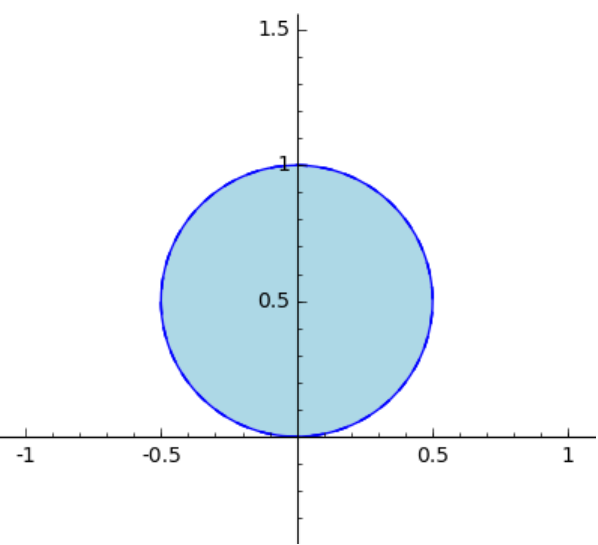}
   	\caption{Domains $S_0=\{z^6+y^2-z^2<0\}$ (left), and $S_2=\left\lbrace y_2^2+z_2^2-z_2<0  \right\rbrace$ (right).}\label{fig:pi_S1}
   \end{figure}
   
   At this step, we notice that the boundary of $S_2$ is in fact a smooth variety whose tangent line at the origin is $z_2=0$. In addition, any other line at the origin intersects the interior of $S_2$. Thus, the strict transform of $S_2$ is unbounded at any chart of an another blow-up.
   
   In order to resolve this situation, we partition $S_3$ in $S_3^1=S_3\cap\{y_2>0\}$ and $S_3^2=S_3\cap\{y_2<0\}$, and one has that $\Ical\left(S_3^1,1/z_3\right)=\Ical\left(S_3^2,1/z_3\right)$ by symmetry. We blow-up $S_3^1$ taking the chart with respect the line $z_3+y_3=0$, obtaining that:   
   \[
       \int_{S_3^1}\frac{\dint y_3\dint z_3}{z_3}=\int_{S_4^1}\frac{\sqrt{2}\dint y_4\dint z_4}{1+z_4},
   \]
   with $S_4^1=\left\lbrace y_4>0, 1-z_4>0, -y_4z_4^2-y_4 + \frac{\sqrt{2}}{2}(z_4+1) \right\rbrace$, pictured  in Figure~\ref{fig:pi_S41_S51}. It remains to resolve the pole of order $1$ at the point $(0,-1)$. Locally, the tangent cone has equations $2y_4-\frac{2}{2}z_4=0$ at this point. Since $S_4^1\subset \{z_4+1>0\}$, we take the chart with respect to the line $z_4+1=0$ to obtain a regular rational form:
   \[
       \int_{S_4^1}\frac{\sqrt{2}\dint y_4\dint z_4}{1+z_4}=\int_{S_5^1}\sqrt{2}\dint y_5\dint z_5,
   \]
   with $S_5^1=\left\lbrace y_5>0,-1<z_5<1,y_5(1+z_5^2)<\sqrt{2}/2\right\rbrace$ (see Figure~\ref{fig:pi_S41_S51}). Repeating this process with $S_3^2$, we obtain an identical piece $S_5^2$, symmetric to $S_5^1$ with respect to the OZ-axis. In fact,  %
   we obtain the same semi-canonical reduction (up to isometry) as in Example~\ref{ex:pi1_dim2}, thus $\eta=\pi$.
   \begin{figure}[ht]
   	\centering
   	\includegraphics[scale=0.45]{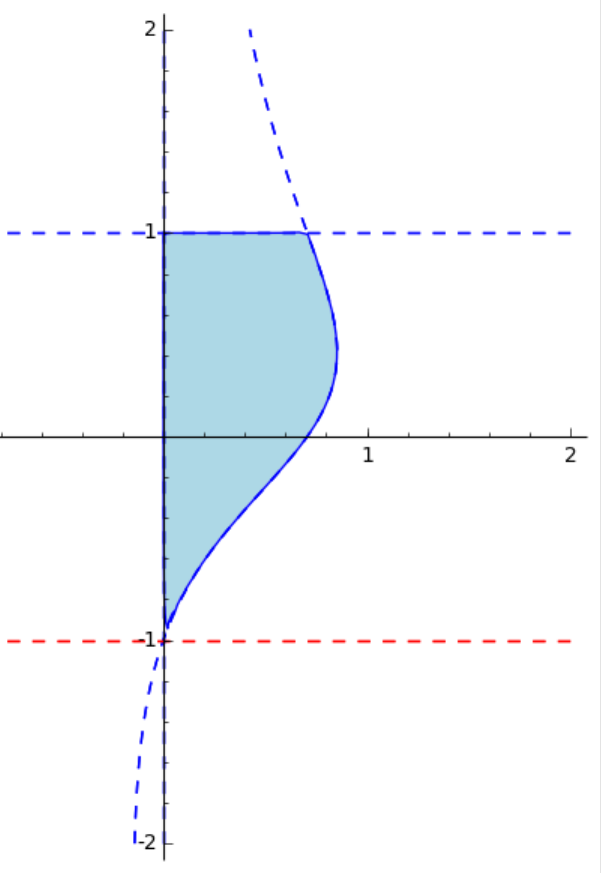}\hspace*{1cm}\includegraphics[scale=0.45]{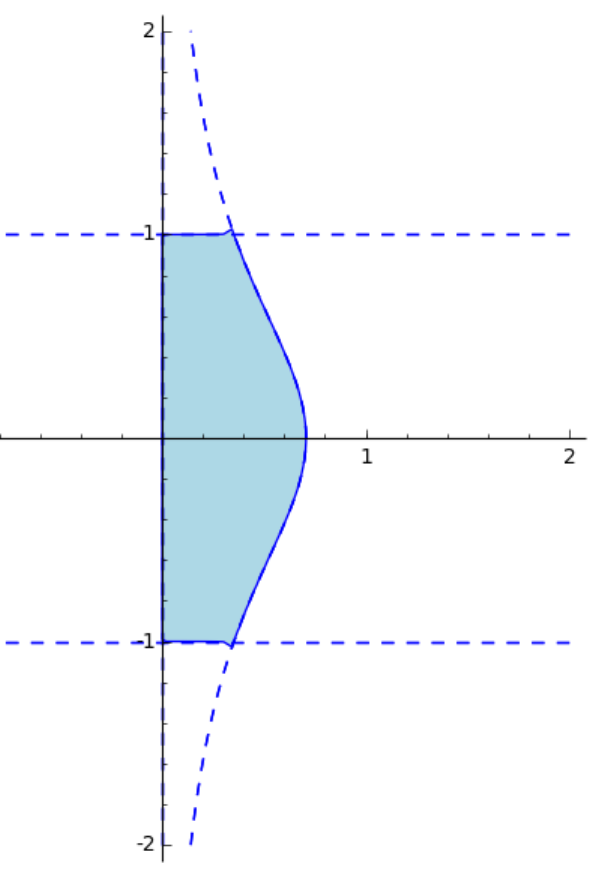}
   	\caption{The domains $S_4^1=\left\lbrace y_4>0, 1-z_4>0, -y_4z_4^2-y_4 + \frac{\sqrt{2}}{2}(z_4+1) \right\rbrace$ including the pole locus in red (left) and $S_5^1=\left\lbrace y_5>0,-1<z_5<1,y_5(1+z_5^2)<\frac{\sqrt{2}}{2}\right\rbrace$ (right).}\label{fig:pi_S41_S51}
   \end{figure}
\end{ex}

\subsection{Multiple zeta values}\label{subsec:mzv}

We have previously introduced multiple zeta values $\zeta(s_1,\ldots,s_k)$ as examples of real periods. This numbers are also described as iterated integrals which can be expressed as the integral of a rational function which depends on the tuple $(s_1,\ldots,s_k)$ over a simplex $\triangle$ of dimension $k+1$, see e.g.~\cite[Sec.~2]{Wald00} for more details.

\begin{ex}%
    Consider the value
    \[
    	\zeta(2)=\sum_{n\geq 1} \frac{1}{n^2}=\frac{\pi^2}{6}.
    \]
    We know that it can be expressed as the integral
    \[
        \int_{\triangle} \frac{\dint x\dint y}{(1-x)y},
    \]
    over the open simplex $\triangle=\left\lbrace 0< x< y< 1\right\rbrace$. The above denominator gives two poles in $\partial\triangle$, one at the origin and another at $(1,1)$. Note that the tangent cone of $\partial\triangle$ at a point $p\in\partial\triangle$ is given by the lines containing the faces involving $p$. After a first blow-up at the origin and taking the second chart $\phi(x_1,y_1)=(x_1y_1,y_1)$, one has that:
    \[
        \int_{\triangle} \frac{\dint x\dint y}{(1-x)y}=\int_{\Box} \frac{\dint x_1\dint y_1}{1-x_1y_1},
    \]
    where $\Box=\phi^{-1}\triangle=\{-1< x,y < 1\}$ (see Figure~\ref{fig:pi2_TC}).
    
    \begin{figure}[ht]
   	\centering
   	\includegraphics[scale=0.45]{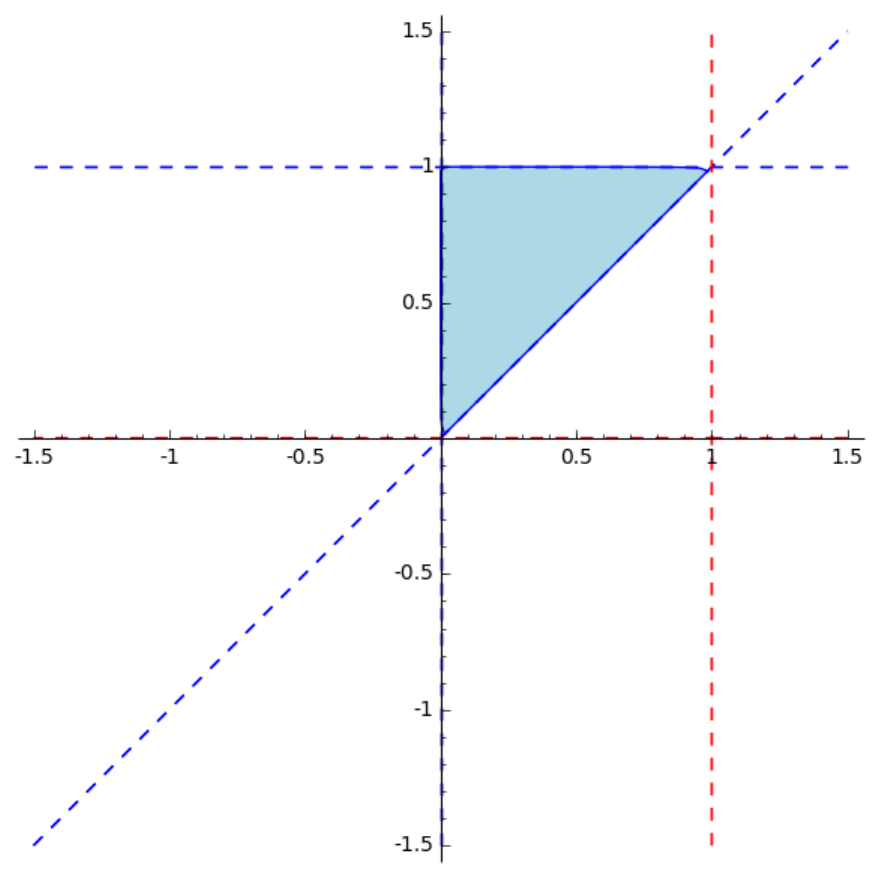}\hspace*{1cm}\includegraphics[scale=0.45]{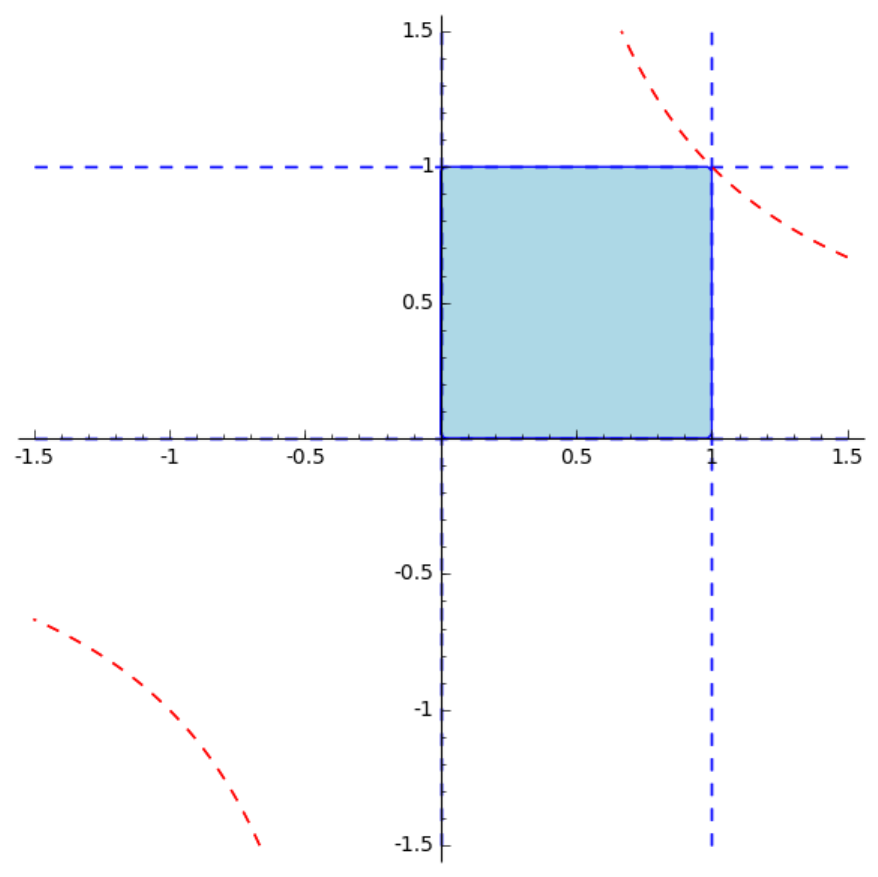}
   	\caption{The domains $\triangle=\left\lbrace 0< x< y< 1\right\rbrace$ (left), and $\Box=\{-1< x,y < 1\}$ (right), including the pole locus in red.}\label{fig:pi2_TC}
   \end{figure}
   
   The tangent cone of $\partial\Box$ at the remaining pole is exactly the translated coordinate axis at $(1,1)$. Thus, we finish the procedure taking the chart of the blow-up with respect to the line $L=\{x_1+y_1-2=0\}$. We can construct such a map $\phi$ by composing the blow-up at the origin with the isometry which sends the origin to $(1,1)$ and the line ${y_1=0}$ to ${x_1+y_1-2=0}$. One has that $\phi$ is an isomorphism between $\RR^2\setminus \{x_2=0\}$ and $\RR^2\setminus L$, inducing the equality:
   \[
       \int_{\Box} \frac{\dint x_1\dint y_1}{1-x_1y_1}=\int_{T}\frac{2\dint x_2\dint y_2}{-x_2y_2^2 +  x_2 + 2 \sqrt{2}}
   \]
   with $T=\left\lbrace x<0, -1<y_2<1, -x_2y_2 +  x_2 + \sqrt{2}>0, x_2y_2 + x + \sqrt{2}>0 \right\rbrace$, which does not contain poles on the boundary, see Figure~\ref{fig:pi2_T_Tf}. Furthermore, the rational function $f(x_2,y_2)=2/(-x_2y_2^2 +  x_2 + 2 \sqrt{2})$ does not change of sign over $T$, then taking the volume under the hypersurface $f=0$, one has:
   \[
       \zeta(2)=\vol_3(T_f),\quad\text{with}\quad T_f=\left\lbrace (x,y,z)\in T\times\RR \mid z>0, 2+z(xy^2 -  x - 2 \sqrt{2})>0\right\rbrace.
   \]

   \begin{figure}[ht]
   	\centering
   	\includegraphics[scale=0.45]{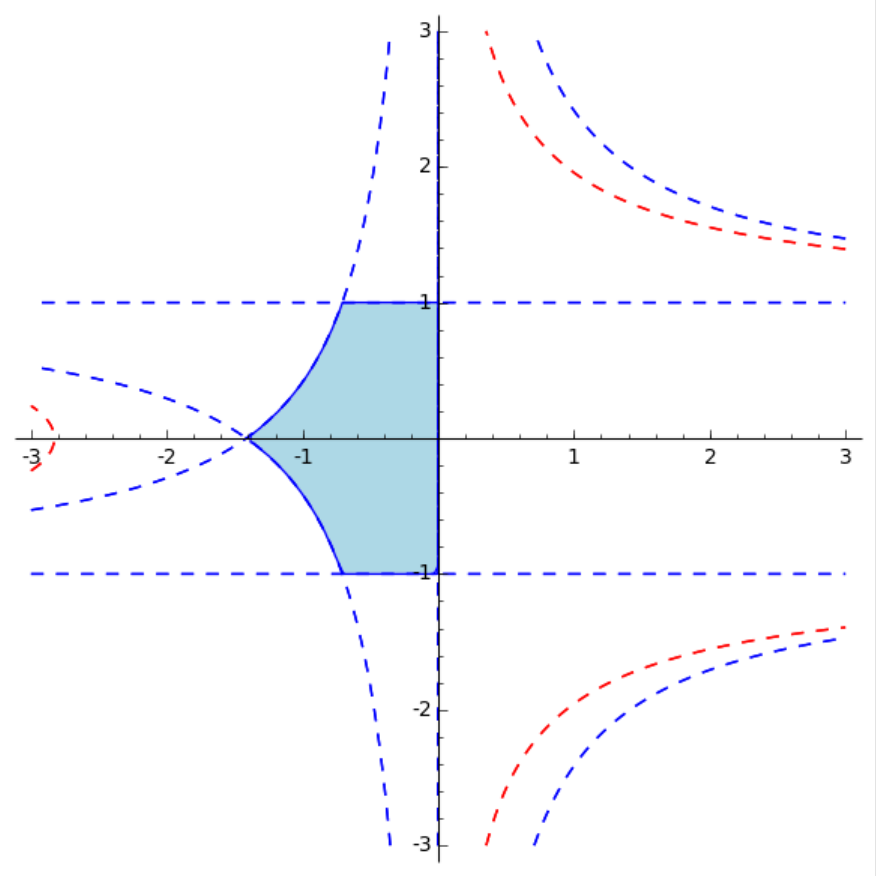}\hspace*{1cm}\includegraphics[scale=1.9]{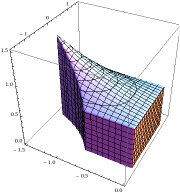}
   	\caption{The domains $T=\left\lbrace x<0, -1<y_2<1, -x_2y_2 +  x_2 + \sqrt{2}>0, x_2y_2 + x + \sqrt{2}>0 \right\rbrace$ including the pole locus in red (left) and $T_f$ (right).}\label{fig:pi2_T_Tf}
   \end{figure}
\end{ex}

\bigskip
\section{Conclusions}\label{sec:conclusions}

\addtocontents{toc}{\protect\setcounter{tocdepth}{1}}

\subsection*{Exponential periods}
Prof.~Waldschmidt asked us about a possible extension of our results in the case of \emph{exponential periods}, i.e. numbers which can be written as values of absolutely convergent integrals of the product of an algebraic function with the exponential of another algebraic function over a semi-algebraic set where all polynomials appearing in the integral have algebraic coefficients. A typical example is
\[
	\sqrt{\pi}=\int_{-\infty}^{+\infty}e^{-x^2}\dint x.
\]
It seems possible, using the same techniques, to find a reduction of exponential periods considering the exponential part as a volume form and generalizing our procedure over the non-exponential part, i.e. a reduction of the form
\[
	\int_{K} e^{g(x_1,\ldots,x_d)}\cdot\dint x_1\wedge\cdots\wedge\dint x_d,
\]
where $K\subset\RR^d$ is a compact semi-algebraic set and $g\in\RRalg\left(x_1,\ldots,x_d\right)$.

\subsection*{Approximation of periods} 
Theorem~\ref{thm:semialg_form_thm} suggests that one could derive a rational or algebraic approximation of a period by computing the volume of a geometric approximation of the compact semi-algebraic set obtained by the reduction algorithm, improving Yoshinaga's work.%

\subsection*{Zero-detection problem}
Prof.~Rivoal asked us about the possibility of detecting the zero as a period using the semi-canonical reduction, i.e. whenever we have an integral $\Ical(S,P/Q)$, to test if this integral is zero or not. The answer is negative, because one has to assume that the volumes of the two compact semi-algebraic sets which express the period by their difference are not equal, as it is detailed in the procedure in Section~\ref{sec:diff_vol}. In fact, the above problem in our setting is equivalent to construct an Equality algorithm for periods.\\

\bigskip
\appendix

\section{Pseudo-code of the semi-canonical reduction procedure}\label{sec:appendix}

In the following, we describe the main procedure given in pseudo-code, called {\sc SemiCanPeriod}. The input data is a triple $(S,P,Q)$ where $S\subset\RR^d$ is a top-dimensional semi-algebraic set, and $P,Q\in\RRalg[X_1,\ldots,X_d]$ are coprime polynomials. A triple $(S,P,Q)$ is called \emph{admissible} if $\Ical(S,P/Q)$ is absolutely convergent. The following pseudo-code expression
\[
	A/B\gets\texttt{arithmetic operations between rational functions}
\] 
means that we are assigning to $A$ (resp. $B$) the numerator (resp. denominator) of the resulting reduced fraction given in the right-hand side.

The procedures of the different intermediary subroutines {\sc CompactifyDomain}, {\sc ResolvePoles} and {\sc VolumeFromDiffSA}, which are detailed in Sections~\ref{sec:sas_projective} and~\ref{sec:diff_vol}, are described in Algorithm~\ref{alg:compactdoms}, Algorithm~\ref{alg:resolvepoles} and Algorithm~\ref{alg:volfromdiff} respectively.

Following the notations in Section~\ref{sec:sas_projective}, we assume that the Hironaka-Villamayor resolution of singularities procedure returns a (finite) list of lists
\[
	\left\{\	\{(V_{i_1,\ldots,i_k},\Phi_{i_1,\ldots,i_k})\}_{(i_1,\ldots,i_k)\in M_k}
	\ \right\}_{k=1,\ldots,r}
\]
of compatible affine charts of the embedded resolution $\pi=\pi_r\circ\cdots\circ\pi_1:W\to\RR^d$ of $X$ numbered by each blow-up $\pi_k$, where:
\begin{itemize}
	\item $V_{i_1,\ldots,i_k}$ is the affine chart represented by the corresponding ring,
	
	\item $\Phi_{i_1,\ldots,i_k}$ is the birational map representing $\pi_k$ between the charts $V_{i_1,\ldots,i_{k-1},i_k}$ and $V_{i_1,\ldots,i_{k-1}}$.
\end{itemize}
\bigskip

\begin{center}
  \captionof{algorithm}{{\sc SemiCanPeriod}: Semi-canonical reduction of $p=\mathcal{I}(S,P/Q)\in\Pkz^\RR$.}\label{alg:reduction}
	\begin{algorithmic}[1]
		\Require{An admissible triple $(S,P,Q)$.}
		\Ensure{A compact semi-algebraic $K$ with same dimension of $S$ such that $\vol(K)=\mathcal{I}(S,P/Q)$.}
		\Procedure{SemiCanPeriod}{$S,P,Q$}
		\State\Comment{Partition by sign of the integrand}
		\State $S^+\gets\{\xx\in S \mid 0<(P/Q)(\xx)\}$
		\State $S^-\gets\{\xx\in S \mid (P/Q)(\xx)<0\}$
		\State\Comment{Lists of triples $\left(S_i^\pm,P_i^\pm,Q_i^\pm\right)$ with $S_i^\pm$ is bounded}
		\State $L^+\gets\textsc{CompactifyDomain}(S^+,P,Q)$
		\State $L^-\gets\textsc{CompactifyDomain}(S^-,P,Q)$
		\State\Comment{Lists of triples $(\widetilde{S_j}^\pm,\widetilde{P_j}^\pm,\widetilde{Q_j}^\pm)$ with resolved poles at the boundary}
		\State $\widetilde{L}^+, \widetilde{L}^-\gets\{\},\{\}$
		\For{$(S^+,P^+,Q^+)\in L^+$ and $(S^-,P^-,Q^-)\in L^-$}
		\State $\widetilde{L}^+\gets\widetilde{L}^+\cup\textsc{ResolvePoles}(S^+,P^+,Q^+)$
		\State $\widetilde{L}^-\gets\widetilde{L}^-\cup\textsc{ResolvePoles}(S^-,P^-,Q^-)$
		\EndFor
		\State\Comment{We define the compact sets under the integrand}
		\State $K^+, K^-\gets\emptyset,\emptyset$
		\For{$(\widetilde{S}^+,\widetilde{P}^+,\widetilde{Q}^+)\in \widetilde{L}^+$ and $(\widetilde{S}^-,\widetilde{P}^-,\widetilde{Q}^-)\in \widetilde{L}^-$}
		\State $K^+\gets K^+\cup\{(\xx,t)\in S^+\times\RR \mid 0\leq t\leq (P^+/Q^+)(\xx)\}$
		\State $K^-\gets K^-\cup\{(\xx,t)\in S^-\times\RR \mid (P^-/Q^-)(\xx)\leq t\leq 0\}$
		\EndFor
		\State\Comment{We construct the compact set $K$ from $K^+$ and $K^-$ which volume is the difference of these sets}
		\If{$\int_S P/Q>0$}
		\State $K\gets\textsc{VolumeFromDiffSA}(K^+,K^-)$
		\Else
		\State $K\gets\textsc{VolumeFromDiffSA}(K^-,K^+)$
		\EndIf
		\State \textbf{return} $K$\Comment{A compact semi-algebraic set $K$ representing $p$}
		\EndProcedure
	\end{algorithmic}
\end{center}\bigskip

\begin{center}
  \captionof{algorithm}{ {\sc CompactifyDomain}: Partition and compactification of domains.}\label{alg:compactdoms}
	\begin{algorithmic}[1]
		\Require{An admissible triple $(S,P,Q)$.}
		\Ensure{A finite list of admissible triples $(S_i,P_i,Q_i)$ such that $S_i$ compact and $\mathcal{I}(S,P/Q)=\sum_i\mathcal{I}(S_i,P_i/Q_i)$.}
		\Procedure{CompactifyDomain}{$S,P,Q$}
		\State $d\gets\dim S$
		\State $S_0\gets S\cap\{-1\leq x_1\leq 1,\ldots,-1\leq x_d\leq 1\}$
		\State $L_\text{dec}\gets \{(S_0,P,Q)\}$\Comment{List of results with the decomposition}
		\For{$i\gets 1,\ldots,d$}
		\State $V_i\gets \bigcap_{j=1}^d \left\lbrace x_i\geq 1,x_i\geq x_j,x_i\geq -x_j\right\rbrace\cup\left\lbrace x_i\leq -1,x_i\leq x_j,x_i\leq -x_j\right\rbrace$
		\State $S_i\gets S\cap V_i$
		\State $S_i\gets$ Change of variables in $S_i$: $x_i=1/x_0$, $x_j=x_j/x_0,\forall j\neq i$
		\State $P_i/Q_i\gets$ Change of variables in $P_i/Q_i$: $x_i=1/x_0$, $x_j=x_j/x_0,\forall j\neq i$
		\State $P_i/Q_i\gets P_i/Q_i\times(1/x_0^{d+1})$\Comment{The Jacobian of the change of variables}
		\State $L_\text{dec}\gets L_\text{dec}\cup\{(S_i,P_i,Q_i)\}$
		\EndFor
		\State \textbf{return} $L_\text{dec}$
		\EndProcedure
	\end{algorithmic}
\end{center}\bigskip

\begin{center}
  \captionof{algorithm}{ {\sc ResolvePoles}: Resolution of poles on the boundary of the domain.}\label{alg:resolvepoles}
	\begin{algorithmic}[1]
		\Require{An admissible triple $(S,P,Q)$.}
		\Ensure{A finite list of admissible triples $(\widetilde{S}_i,\widetilde{P}_i,\widetilde{Q}_i)$ such that $\widetilde{S}_i$ is compact, $\widetilde{Q}_i$ has not zeros in $\widetilde{S}_i$ and $\mathcal{I}(S,P/Q)=\sum_i\mathcal{I}(\widetilde{S}_i,\widetilde{P}_i/\widetilde{Q}_i)$.}
		\Procedure{ResolvePoles}{$S,P,Q$}
		
		\State $L_\text{dec}\gets \{\}$\Comment{List of results}
		\If{$S\cap \{Q=0\}=\emptyset$}
			\State $L_\text{dec}\gets\{(S,P,Q)\}$
		\Else
			\State $d\gets\dim S$
			\State $X\gets\partial_z S\cup \{P=0\}\cup \{Q=0\}$
			\State $\{
			\{(V_{i_1,\ldots,i_k},\Phi_{i_1,\ldots,i_k})\}_{(i_1,\ldots,i_k)\in M_k}
			\}_{k=1,\ldots,r}\gets$ The list of lists associated to of the embedded resolution $\pi=\pi_r\circ\cdots\circ\pi_1:W\to\RR^d$ of $X$.
			
			\State $M_\text{stop}\gets \{\}$\Comment{List of indices of charts where the semi-alg. set and pole locus are disjoint}
			\For{$i\gets 1,\ldots,m_1$}\Comment{Initialization: first blow-up $\pi_1:W_1\to\RR^d$}
			\State $\widetilde{S}_{i}\gets
			\textsc{Closure}(\textsc{Interior}(
			\Phi^{-1}_{i}(S)\cap(\RR^d\times C_{i})
			))$
			\State $\widetilde{P}_{i}/\widetilde{Q}_{i}\gets (\Phi_{i})^* (P/Q)\times \Jac(\Phi_{i})$
			\If{$\widetilde{S}_{i}\cap \{Q_{i}=0\}=\emptyset$}
				\State $L_\text{dec}\gets L_\text{dec}\cup\{(\widetilde{S}_{i},\widetilde{P}_{i},\widetilde{Q}_{i})\}$
				\State $M_\text{stop}\gets M_\text{stop}\cup\{(i)\}$
			\EndIf	
			\EndFor
			\For{$k\gets 2,\ldots,r$}\Comment{Recursion: intermediary blow-up $\pi_k:W_k\to W_{k-1}$}
				\For{$(i_1,\ldots,i_k)\in M_k$}
				\If{$(i_1,\ldots,i_{k'})$ is not in $M_\text{stop}$ for any $0\leq k'\leq k$}
					\State $\widetilde{S}_{i_1,\ldots,i_k}\gets
					\textsc{Closure}(\textsc{Interior}(
					\Phi^{-1}_{i_1,\ldots,i_k}(\widetilde{S}_{i_1,\ldots,i_{k-1}})\cap(\RR^d\times C_{i_k})
					))$
					\State $\widetilde{P}_{i_1,\ldots,i_k}/\widetilde{Q}_{i_1,\ldots,i_k}\gets (\Phi_{i_1,\ldots,i_k})^* (\widetilde{P}_{i_1,\ldots,i_{k-1}}/\widetilde{Q}_{i_1,\ldots,i_{k-1}})\times \Jac(\Phi_{i_1,\ldots,i_k})$ 
					\If{$\widetilde{S}_{i_1,\ldots,i_k}\cap \{Q_{i_1,\ldots,i_k}=0\}=\emptyset$}
						\State $L_\text{dec}\gets L_\text{dec}\cup\{(\widetilde{S}_{i_1,\ldots,i_k},\widetilde{P}_{i_1,\ldots,i_k},\widetilde{Q}_{i_1,\ldots,i_k})\}$
						\State $M_\text{stop}\gets M_\text{stop}\cup\{(i_1,\ldots,i_k)\}$
					\EndIf	
				\EndIf
				\EndFor
			\EndFor
		\EndIf
		\State \textbf{return} $L_\text{dec}$
		\EndProcedure
	\end{algorithmic}
\end{center}\bigskip

\begin{center}
  \captionof{algorithm}{ {\sc VolumeFromDiffSA}: Construction of a compact semi-algebraic set from the difference of other two.}\label{alg:volfromdiff}
	\begin{algorithmic}[1]
		\Require{Two compact semi-algebraic sets $K_1,K_2$ of maximal dimension $d$ such that $\vol_d(K_2)<\vol_d(K_1)<+\infty$.}
		\Ensure{A compact semi-algebraic $K$ such that $\dim K=d$ and $\vol_d(K)=\vol_d(K_1)-\vol_d(K_2)$.}
		\Procedure{VolumeFromDiffSA}{$K_1,K_2$}
		\State $r\gets\min\{N\in\NN \mid K_1\cup K_2\subset [0,N]^d\}$
		\State $\Delta_1\gets \{\}$, $\Delta_2\gets\{\}$
		\State $\delta_1\gets 0$, $\delta_2\gets 1$
		\State $n\gets 1$
		\While{$\delta_1<\delta_2$}\Comment{This condition fails in finite time as a consequence of Lemma~\ref{lem:deltas_ineq_inclusion} }
		\For{$(k_1,\ldots,k_d)\in\{0,\ldots,n\}^d$}
		\State $\mathring{C}_n(k_1,\ldots,k_d)\gets \left(\frac{k_1}{n}r,\frac{k_1+1}{n}r\right)\times\ldots\times\left(\frac{k_d}{n}r,\frac{k_d+1}{n}r\right)$
		\If{$\mathring{C}_n(k_1,\ldots,k_d)\subset K_1$}
		\State $\Delta_1\gets\Delta_1\cup\{\mathring{C}_n(k_1,\ldots,k_d)\}$
		\ElsIf{$\mathring{C}_n(k_1,\ldots,k_d)\cap K_2\neq\emptyset$}
		\State $\Delta_2\gets\Delta_2\cup\{\mathring{C}_n(k_1,\ldots,k_d)\}$
		\EndIf
		\State $\delta_1\gets\#\Delta_1$, $\delta_2\gets\#\Delta_2$
		\EndFor  
		\EndWhile
		\State $K\gets K_1$
		\For{$k\gets1,\ldots,\delta_2$}\Comment{Elimination}
		\State $D\gets K_2\cap\Delta_2[k]$
		\State $D\gets$ Change of variables in $D$: $\tilde{x_i}=x_i + \frac{k_i' - k_i}{n}r, \forall x_i$, where $(\Delta_1[k])_i=\left(\frac{k_i'}{n}r,\frac{k_i'+1}{n}r\right)$
		\State $K\gets K\setminus D$
		\EndFor
		\State \textbf{return} \textsc{Closure}($K$)
		\EndProcedure
	\end{algorithmic}
\end{center}
\bigskip
%
%
%
\addtocontents{toc}{\protect\setcounter{tocdepth}{0}}

\bigskip
\section*{Acknowledgments}

The author would like to thank Professors Jacky Cresson and Enrique Artal for their support, encouragement and numerous discussions. We are also very grateful to Professors Michel Waldschmidt and Pierre Cartier for helpful discussions and ideas. The author would also like to thank the anonymous referee for their valuable comments which helped improve the manuscript.

\addtocontents{toc}{\protect\setcounter{tocdepth}{1}}

\bigskip
\bigskip
\bibliographystyle{alpha}
\bibliography{biblio}

\end{document}